\documentclass[a4paper, reqno, 11pt]{amsart}

\usepackage[sc, noBBpl]{mathpazo}	
\usepackage{mathtools}

\usepackage{amsmath, amsthm, amssymb}
\usepackage[all]{xy}
\usepackage[usenames,dvipsnames]{color}
\usepackage[T1]{fontenc}
\usepackage[utf8]{inputenc}
\usepackage{graphicx}
\usepackage{comment}
\usepackage[linkbordercolor={0.8 0.9 0.9}, citebordercolor={0.85 0.85 0.85}]{hyperref}
\usepackage{cleveref}    
\usepackage[normalem]{ulem}

\usepackage[
backend=biber,
style=alphabetic
]{biblatex}
\numberwithin{equation}{section}

\theoremstyle{plain}
\newtheorem{theorem}{Theorem}[section]
\newtheorem*{theorem*}{Theorem}
\newtheorem{proposition}[theorem]{Proposition}
\newtheorem{lemma}[theorem]{Lemma}
\newtheorem{corollary}[theorem]{Corollary}
\theoremstyle{definition}
\newtheorem{definition}[theorem]{Definition}
\newtheorem{example}[theorem]{Example}
\theoremstyle{remark}
\newtheorem{remark}[theorem]{Remark}

  {\par\begin{tabular}{rcl}}%
  {\end{tabular}\par}

\newcommand{\ds}{\displaystyle}

\newcommand{\ip}[1]{\langle #1 \rangle}

\newcommand{\into}{\hookrightarrow}

\makeatletter
\newcommand*\slot{\mathpalette\slot@{.5}}
\newcommand*\slot@[2]{\mathbin{\vcenter{\hbox{\scalebox{#2}{$\m@th#1\bullet$}}}}}
\makeatother

\newcommand{\lie}{\mathfrak}

\newcommand{\CC}{\mathbb{C}}

\newcommand{\NN}{\mathbb{N}}

\newcommand{\TT}{\mathbb{T}}

\newcommand{\ZZ}{\mathbb{Z}}

\newcommand{\cA}{\mathcal{A}}
\newcommand{\cB}{\mathcal{B}}

\newcommand{\cG}{\mathcal{G}}

\newcommand{\cO}{\mathcal{O}}

\newcommand{\cU}{\mathcal{U}}

\newcommand{\rmc}{\mathrm{c}}

\newcommand{\SU}{\mathrm{SU}}

\newcommand{\bP}{\mathbf{P}}		
\newcommand{\bQ}{\mathbf{Q}}		

\newcommand{\Cc}{C_\rmc}

\DeclareMathOperator{\supp}{supp}

\DeclareMathOperator{\id}{id}

\DeclareMathOperator{\Vect}{span}
\DeclareMathOperator{\CCE}{CCE}

\newcommand{\sfS}{\mathsf{S}}
\newcommand{\sfR}{\mathsf{R}}

\newcommand{\crystal}{\overleftarrow{\cB}\!}
\DeclareMathOperator{\wt}{wt}
\DeclareMathOperator{\rk}{rk}
\newcommand{\Simples}{\mathbf{\Delta}}

\newcommand{\bPI}{\bP_I}
\newcommand{\bPIplus}{\bP^+_I}
\newcommand{\GK}{\cG(K)}
\newcommand{\GYI}{\cG(Y^I)}
\newcommand{\GXI}{\cG(X^I)}
\newcommand{\GXIL}{\cG(X^{I,L})}
\newcommand{\GIK}{\cG_I(K)}

\title[The crystal groupoid of a compact semisimple Lie group]{The crystal groupoid of a compact semisimple Lie group and its flag varieties}
\author{Aidan Sims}
\address{School of Mathematics and Statistics\\
UNSW Sydney\\
NSW 2052\\
Australia}
\email{aidan.sims@unsw.edu.au}

\author{Robert Yuncken}
\address{Universit\'e de Lorraine\\
CNRS\\
IECL\\
F-57000 Metz\\
France}
\email{robert.yuncken@univ-lorraine.fr}

\thanks{A.~Sims was supported by ARC grant DP250100297. R.~Yuncken was supported by ANR grants OpART (ANR-23-CE40-0016) and CroCQG (ANR-25-CE40-5010). The article is based upon work from COST Action CaLISTA CA21109 supported by COST (European Cooperation in Science and Technology): www.cost.eu.  R.~Yuncken is extremely grateful for the support and hospitality of the Sydney Mathematical Research Institute (SMRI) and the University of Wollongong.} 


\date{\today}

\begin{document}

\begin{abstract}
  We introduce the crystal groupoid of a compact semisimple Lie group $K$; it is the Deaconu-Renault groupoid of a higher-rank analogue of a subshift of finite type on an inverse limit of Kashiwara crystals.  We show that its Steinberg algebra and its $C^*$-algebra realise the crystal limits of the polynomial function algebra of $K$ and its $C^*$-envelope, respectively, in the sense of Matassa and the second author. We show that the orbit space of the crystal groupoid is in bijection with the Weyl group of $K$, with topology determined by the Bruhat order on the Weyl group. We also define a crystal groupoid for flag manifolds and related Poisson homogeneous spaces, and prove the analogous structural results.
\end{abstract}

\maketitle

\section{Introduction}

Let $K$ be a simply connected compact semisimple Lie group.  Kashiwara and Lusztig showed that the representation theory of $K$ simplifies to purely combinatorial data if one passes to the quantised enveloping algebra $\cU_q(\lie{k}_\CC)$ and then takes the crystal limit as $q\to 0$.  The Kashiwara renormalisations $e_i$, $f_i$ of the simple root vectors act as partial isometries, preserving a ``crystal basis'' for each representation.

A similar thing happens for the coordinate function algebra of the group $\cO(K)$.  This too admits a $q$-deformation, now to a family of compact quantum group algebras $\cO(K_q)$ with $0<q<\infty$, and it was shown in \cite{MatYun:crystals} that $\cO(K_q)$ also admits a limit at $q=0$, generalising earlier examples from \cite{Woronowicz:pseudogroups,HonSzy:lens,Giselsson:SU3,GirPal}.
Moreover, the matrix coefficients in $\cO(K_0)$ corresponding to the crystal basis elements are all partial isometries.

Explicitly, the crystal limit, $\cO(K_0)$, is a Kumjian-Pask algebra \cite{Aranda:Kumjian-Pask}.  For those who prefer $C^*$-algebras, the $C^*$-envelope $C(K_0)$ of ``continuous functions'' is a higher-rank graph algebra \cite{KumPas}.  Although they are defined via generators and relations, higher-rank graph $C^*$-algebras are known to be isomorphic to $C^*$-algebras of Deaconu-Renault groupoids of semigroup actions \cite{KumPas}.  For more on Deaconu-Renault groupoids as models for Cuntz-Krieger-like algebras,  see \cite{Deaconu,Renault:groupoid_approach,KumPasRaeRen,KumPasRae,Paterson,Renault:Cuntz-like}.
It is natural to ask what the groupoid is for the crystallised function algebras.

Groupoid models for the function algebras of quantum $\SU(n)$, quantum spheres and quantum projective spaces were obtained by Sheu in the 1990s \cite{Sheu:CQGs,Sheu:CQG_lectures,Sheu:spheres,Sheu:projective_spaces} by using explicit calculations with Vaksman and Soibelman's representation theory \cite{VakSoi:SUn,Soibelman}.  
In this article, we take a more representation theoretic approach, using the results of \cite{MatYun:crystals} to describe a simple groupoid model of the crystallised function algebra of any compact semisimple Lie group, as well as any flag manifold and any of the related Poisson homogeneous spaces of the types considered in \emph{e.g.} \cite{StoDij,Stokman:quantum_orbit_method,NesTus:functions}.  

The resulting groupoids are explicitly related to Kashiwara's crystal basis theory.  For a dominant weight $\lambda\in\bP^+$ of $K$, let $\cB(\lambda)$ denote the irreducible crystal of highest weight $\lambda$.
We define the \emph{semi-infinite crystal} to be the inverse limit
 \begin{equation}
 \label{eq:crystal_intro}
   \crystal(\infty) = \underset{\lambda\in\bP^+}{\underleftarrow{\lim}}\cB(\lambda)
 \end{equation}
over all finite crystals, with connecting maps
\begin{equation}
  \label{eq:R_intro}
  \sfR_\mu : \cB(\lambda) \into \cB(\lambda-\mu)\otimes\cB(\mu) \to \cB(\mu),
  \qquad \text{for $\lambda\geq\mu$,}
\end{equation}
given by including $\cB(\lambda)$ into the tensor product then projecting onto the second tensor leg.  (Note that $\crystal(\infty)$ is not to be confused with Kashiwara's infinite crystal $\cB(\infty)$ from \cite{Kashiwara:crystal2}.)

The semi-infinite crystal $\crystal(\infty)$ admits a semigroup action of the dominant weight lattice
\begin{equation}
  \sfS : \bP^+ \curvearrowright \crystal(\infty),
\end{equation}
called the \emph{shift action}, which is the inverse limit of the projections onto the first tensor leg in~\eqref{eq:R_intro}, see Section~\ref{sec:action}.  This construction is analogous to the construction of a subshift of finite type as an inverse limit of finite words with the shift action.

The \emph{crystal groupoid}
\begin{equation}
\GK \rightrightarrows\crystal(\infty)
\end{equation}
is the Deaconu-Renault groupoid of this semigroup action \cite{Deaconu,Renault:Cuntz-like}, \emph{i.e.},
\[
  \GK = \{(x,\lambda,y) \in \crystal(\infty)\times\bP\times\crystal(\infty) \mid \text{$x$ and $y$ are tail equivalent} \},
\]
see Definition \ref{def:crystal_groupoid}.  The space $\Cc^\infty(\GK)$ of locally constant compactly supported functions on $\GK$ is a $*$-algebra under convolution; these functions may be considered ``smooth'' in the same way as locally constant functions on a $p$-adic group.  The $C^*$-completion
is denoted $C^*(\GK)$.  The crystal groupoid is amenable, so there is no ambiguity here.

Our main result is the following.

\begin{theorem*}[see Theorem~\ref{thm:isomorphism}]
  There are $^*$-isomorphisms
  \[
    \cO(K_0) \cong \Cc^\infty(\GK)\quad\text{ and }\quad C(K_0)\cong C^*(\GK).
  \]
\end{theorem*}

We also prove an analogue of this result for non-simply connected $K$, and for flag manifolds and other Poisson homogeneous spaces \cite{StoDij,Stokman:quantum_orbit_method,NesTus:functions}, see Theorem~\ref{thm:isomorphism2}.  Specifically, for any $X$ in the family of Poisson homogeneous spaces of \cite{NesTus:functions}, there is a crystal groupoid $\cG(X)$ such that we have $*$-isomorphisms
\[
    \cO(X_0) \cong \Cc^\infty(\cG(X))\quad\text{ and }\quad C(X_0)\cong C^*(\cG(X)).
\]

Finally, we describe the orbit space of these crystal groupoids as topological spaces.  We show that the orbit space of $\GK$ is homeomorphic to the Weyl group $W$ with topology such that the closed sets are the order ideals for the Bruhat order.  For a homogeneous space $X$, the orbit space of $\cG(X)$ is $W/W_I$ for an appropriate subgroup $W_I$, again with the Bruhat order topology.

\medskip

The article is structured as follows.  In Section~\ref{sec:crystal}, we define the semi-infinite crystal $\crystal(\infty)$.  We also define the semigroup action of the dominant weights $\bP^+ \cong \NN^{\rk(K)}$ on $\crystal(\infty)$ and the associated Deaconu-Renault groupoid $\GK$, here called the {crystal groupoid}.  In Section~\ref{sec:algebra}, we prove the isomorphism theorem above.  In Section~\ref{sec:structure}, we examine the structure of the crystal groupoid and of its orbit space.  In Section~\ref{sec:homogeneous spaces} we define the crystal groupoid $\cG(X)$ for flag manifolds $X$ and  other Poisson homogeneous spaces.  We establish the $*$-isomorphisms above and describe the structure of the groupoid orbit spaces.

\subsection*{Acknowledgments}

It is a pleasure to thank J.~Renault, M.~Matassa,\linebreak S.~Neshveyev and T.~Schulz for discussions on this material. 

\subsection*{A.I. disclosure}

We used ChatGPT in the proof-reading stage of this manuscript. In addition to identifying a number of typos, it suggested some additional references and pointed out that we had forgotten to prove injectivity of the map $\pi_I^*$ in Lemma~\ref{lem:pi_I_star}.   A.I. was not used to develop any results or proofs.

\subsection*{Conventions and notation}

We use $\NN = \{0,1,2,\ldots\}$.


\section{The semi-infinite crystal}
\label{sec:crystal}

\subsection{Finite crystals}

Let $K$ be a simply connected compact semisimple Lie group and $T$ a maximal torus. Writing $r=\rk(K)$ for the rank of $K$, the weight lattice of $K$ is $\bP=\widehat{T} \cong \ZZ^{r}$.   We fix a set $\Simples = \{\alpha_1,\ldots,\alpha_r\}$ of simple roots for the complexified Lie algebra $\lie{k}_\CC$, and let $\{\varpi_1,\ldots,\varpi_r\}$ be the corresponding  fundamental weights.  We often silently make the identification $\Simples = \{1,\ldots,r\}$.  The set of dominant integral weights is $\bP^+ = \NN\{\varpi_i\} \cong \NN^r$.  It is a semigroup and endows $\bP$ with an ordering.

We will recall just enough crystal basis theory to be able to follow the article.  We will shortly convert this theory into symbolic dynamics, at which point only the formal properties of crystal bases will be needed.  For more details on crystal bases, we refer the reader to \cite{HonKan} for a representation-theoretic approach or \cite{BumpSchilling} for a combinatorial approach.

A finite-dimensional unitary representation $V$ of $K$ gives rise to a $*$-representation of the universal enveloping algebra $\cU(\lie{k}_{\CC})$.  This can then be deformed to a $*$-representation of the quantised enveloping algebra $\cU_q(\lie{k}_\CC)$ on $V$ for $0<q<1$.  In the crystal limit, as $q\to0$, one obtains a basis $\cB$ of $V$, called a \emph{crystal}, which is stable under the action of renormalised versions $e_i,f_i$ of the Chevalley generators $E_i,F_i\in\cU_q(\lie{k}_\CC)$.

Let us extract the formal properties.  A (finite) crystal is a finite set $\cB$ equipped with:
\begin{itemize}
\item a pair of maps $e_i,f_i : \cB \sqcup \{0\} \to \cB \sqcup \{0\}$, called \emph{Kashiwara operators}, such that $e_i(0) = f_i(0) = 0$ and
\begin{equation}
  \label{eq:Kashiwara}
    e_i(a) = b \iff f_i(b)=a, \qquad \text{whenever }a,b\neq 0,
\end{equation}
\item a map $\wt:\cB \to \bP$, called the \emph{weight}, such that
\[
  \wt(e_i(b)) = \wt(b) + \alpha_i, \qquad \text{whenever }e_i(b) \neq 0,
\]
\end{itemize}
as well as some other axioms which we will ignore for now; see \cite[\S{}2.2]{BumpSchilling} or \cite[\S{}4.2]{HonKan} for the precise axioms.

The Kashiwara operators define a graph structure on $\cB$, where two elements of $\cB$ are joined by an edge if and only if they are connected by a pair of Kashiwara operators as in~\eqref{eq:Kashiwara}.  The connected components of the crystal graph correspond to the irreducible subrepresentations of $V$, and we call them irreducible subcrystals.  For each $\mu\in\bP^+$, there is an irreducible crystal of highest weight $\mu\in\bP^+$, which we will denote by $\cB(\mu)$.

Crucially, crystals are compatible with tensor products, in the sense that if $\cB$ and $\cB'$ are crystals for $V$ and $V'$, respectively, then $\cB\otimes\cB'$ is a crystal for $V\otimes V'$.  Note that $\cB\otimes\cB'$ is really just the Cartesian product $\cB\times\cB'$, but we will continue to use the suggestive tensor product notation.
The action of the Kashiwara operators on the tensor product crystal is given by Kashiwara's tensor rule, see \cite[\S{}2.3]{BumpSchilling} or \cite[\S{}4.4]{HonKan}.

\subsection{The semi-infinite crystal}

\begin{definition}
\label{def:Cartan_component}
 Let $\mu,\nu\in\bP^+$.
  The \emph{Cartan component} of $\cB(\mu)\otimes\cB(\nu)$ is the unique irreducible subcrystal of highest weight $\mu+\nu$.
  More generally, for $\mu_n,\ldots,\mu_0\in\bP^+$, the Cartan component of $\cB(\mu_n)\otimes\cdots\otimes\cB(\mu_0)$ is the unique irreducible subcrystal of highest weight $\sum_{j=0}^n\mu_j$.
\end{definition}

Let $\lambda,\mu\in\bP^+$ with $\mu\leq\lambda$.  As mentioned in the introduction, we can use the inclusion of the Cartan component
\[
  \cB(\lambda) \into \cB(\lambda-\mu) \otimes \cB(\mu)
\]
to decompose any $a\in\cB(\lambda)$ as a product
\[
  a \mapsto b\otimes c \qquad\text{with } b\in\cB(\lambda-\mu),~c\in\cB(\mu).
\]
We thus obtain two maps
\begin{align*}
  &\sfR_\mu : \cB(\lambda) \to \cB(\mu), && a\mapsto c,\\
  &\sfS_\mu : \cB(\lambda) \to \cB(\lambda-\mu), \quad && a\mapsto b.
\end{align*}
In \cite{MatYun:crystals}, these maps were called the right-end and left-end maps, respectively.
However, as described in the introduction, we will use $\sfR_\mu$ to define semi-infinite words (infinite to the left, finite to the right) and $\sfS_\mu$ to define a symbolic dynamical system on these.  As such it is appropriate to call $\sfR_\mu$ the \emph{truncation map} and $\sfS_\mu$ the \emph{shift map}.

\begin{lemma}
The truncation and shift maps satisfy the following relations, whenever $\mu+\nu\leq\lambda$:
\begin{align}
   & \sfR_\nu \circ \sfR_{\mu+\nu} = \sfR_{\nu} :  \cB(\lambda) \to \cB(\nu),
   \label{eq:RR} \\
   & \sfS_\nu \circ \sfS_{\mu} = \sfS_{\mu+\nu} :  \cB(\lambda) \to \cB(\lambda-\mu-\nu),
   \label{eq:SS} \\
   & \sfS_{\mu}\circ\sfR_{\mu+\nu} = \sfR_{\nu}\circ\sfS_{\mu} : \cB(\lambda) \to \cB(\nu).
   \label{eq:SR}
\end{align}
\end{lemma}

\begin{proof}
One can directly check these relations on an element $b_2\otimes b_1\otimes b_0$ in the Cartan component of $\cB(\lambda-\mu-\nu) \otimes \cB(\mu)\otimes\cB(\nu)$.
\end{proof}

Thanks to the relation~\eqref{eq:RR}, the set of finite crystals $\cB(\lambda)$ with connecting maps
\begin{equation}
\label{eq:directed_system}
  \sfR_\mu : \cB(\lambda) \to \cB(\mu), \qquad\qquad (\lambda\geq\mu)
\end{equation}
forms a projective system of finite sets.

\begin{definition}
    \label{def:B_infty}
    The \emph{semi-infinite crystal} $\crystal(\infty)$ is the inverse limit of the system of finite crystals,
    \[
      \crystal(\infty) = \underset{\lambda\in\bP^+}{\underleftarrow{\lim}}\cB(\lambda),
    \]
    with connecting maps~\eqref{eq:directed_system}.
    This means that an element of $\crystal(\infty)$ is a family $(b_\lambda)_{\lambda\in\bP^+}$  with $b_\lambda\in\cB(\lambda)$ for every $\lambda$ and such that $\sfR_\mu(b_{\lambda}) = b_\mu$ for all $\lambda\geq\mu$.

    We equip $\crystal(\infty)$ with the inverse limit topology.  So a base of clopen sets is the collection of cylinder sets $\sfR_\mu^{-1}(a)$ indexed by $\mu\in\bP^+$ and $a\in\cB(\mu)$.
\end{definition}

\subsection{The shift action}
\label{sec:action}

The relations~\eqref{eq:RR} and~\eqref{eq:SR} show the compatibility of the truncation and shift maps with the directed system of finite crystals.  We get well-defined truncation and shift maps on the infinite crystal:
    \begin{align}
        \sfR_\mu : &\crystal(\infty) \to \cB(\mu);
        & (b_\lambda)_{\lambda\in\bP^+} &\mapsto b_\mu,
        \label{eq:R}\\
        \sfS_\mu :&\crystal(\infty) \to \crystal(\infty);
        & (b_\lambda)_{\lambda\in\bP^+} &\mapsto (\sfS_\mu(b_{\lambda+\mu}))_{\lambda\in\bP^+}.
    \end{align}
These maps are continuous, and the $\sfS_\mu$ are local homeomorphisms.  Moreover, they satisfy the analogues of relations \eqref{eq:RR}--\eqref{eq:SR} with $\cB(\lambda)$ replaced by $\crystal(\infty)$ and the convention that $\infty-\mu = \infty$ for any $\mu\in\bP^+$.

In particular, the semi-infinite analogue of Equation~\eqref{eq:SS} implies that the shift maps $(\sfS_\mu)_{\mu\in\bP^+}$ define the continuous semigroup action
\begin{equation}
\label{eq:shift}
  \sfS : \bP^+ \curvearrowright \crystal(\infty)
\end{equation}
as mentioned in the introduction.

\subsection{The crystal groupoid}

Now, following Renault \cite{Renault:groupoid_approach} and Deaconu \cite{Deaconu}, we can associate a groupoid to the dynamical system~\eqref{eq:shift}.  The definition is as follows.

\begin{definition}
    \label{def:crystal_groupoid}
    The \emph{crystal groupoid} of the simply connected compact semisimple Lie group $K$ is the groupoid $\GK \rightrightarrows \crystal(\infty)$ given by
    \[
        \GK = \{ (x,\mu-\nu,y) \in \crystal(\infty) \times \bP \times \crystal(\infty) \mid  \mu,\nu\in\bP^+,\ \sfS_{\mu}(x) = \sfS_{\nu}(y) \}
    \]
    with range, source, product and inverse given by
    \begin{align*}
     & r(x,\mu,y) = x, &
     & s(x,\mu,y) = y, \\
     & (x,\mu,y)(y,\nu,z) = (x, \mu+\nu, z), &
     & (x,\mu,y)^{-1} = (y,-\mu,x).
    \end{align*}
\end{definition}

This groupoid comes with a standard topology, as follows.

\begin{lemma}
  \label{lem:topology}
  For any finite crystal elements $a\in\cB(\mu)$, $b\in\cB(\nu)$, put
  \begin{equation}
  \label{eq:Z}
    Z(a,b) = \{(za,\mu-\nu,zb) \in\GK \mid z\in\!\crystal(\infty),\ za, zb\in\!\crystal(\infty)  \}.
  \end{equation}
  Then these sets form a basis of compact open bisections for $\GK$, making it into a second countable étale groupoid.
\end{lemma}

The proof can be found in many sources; see \emph{e.g.}, \cite{SimWil:primitive_ideals}, although note that the topology is described there via a family of basic sets which is different from but equivalent to the family we use here.  We also have the following.

\begin{theorem}[{\cite[Lemma~3.5]{SimWil:primitive_ideals}}]
The groupoid $\GK$ is amenable.
\end{theorem}

\subsection{Semi-infinite crystals correspond to semi-infinite Cartan words}
\label{sec:words}

In order to underscore the relation of the above dynamical system with symbolic dynamics, let us reinterpret the semi-infinite crystal $\crystal(\infty)$ as a space of semi-infinite words in an alphabet given by finite crystals.  This point of view will be useful at many points.

\begin{definition}
  \label{def:Cartan_words}
  Let $(\mu_j)_{j=0}^\infty$ be a sequence of dominant weights.  We will be particularly interested in the case where all the $\mu_j$ are equal, and are strictly dominant, but we do not want to limit ourselves to that situation.

  A finite word
  \begin{equation}
  \label{eq:finite_word}
    x_nx_{n-1} \ldots x_2x_1x_0 \ \in\  \prod_{j=0}^n \cB(\mu_j)
  \end{equation}
  will be called a \emph{Cartan word} if $x_n\otimes\cdots\otimes x_0$ belongs to the Cartan component of $\cB(\mu_n)\otimes\cdots\otimes\cB(\mu_0)$.
  A semi-infinite word
  \begin{equation}
  \label{eq:semi-infinite_word}
    x= \cdots x_3x_2x_1x_0 \ \in\ \prod_{j=0}^\infty \cB(\mu_j)
  \end{equation}
  will be called a \emph{semi-infinite Cartan word} if every finite subword is a Cartan word.
  We will write $\emptyset$ for the unique element of the trivial crystal $\cB(0)$, and view it as the empty word.
\end{definition}

We can identify finite Cartan words with finite crystal elements.  Specifically, if $\mu=\sum_{j=0}^n \mu_j$, then for a Cartan word as in~\eqref{eq:finite_word}, we write
\[
  x=x_n\ldots x_1 x_0 \ \in \ \cB(\mu),
\]
when $x\in\cB(\mu)$ corresponds to $x_n\otimes\cdots\otimes x_0$ under the inclusion of the Cartan component $\cB(\mu) \into \cB(\mu_n) \otimes\cdots\otimes\cB(\mu_0)$.  Any $x\in\cB(\mu)$ decomposes uniquely as a finite Cartan word in this way, and the components are given by
\begin{equation}
  \label{eq:letter}
  x_k = \sfR_{\mu_k}\circ\sfS_{\sum_{j=0}^{k-1}\mu_j}(x).
\end{equation}

Similarly, we can identify semi-infinite Cartan words with elements of the semi-infinite crystal $\crystal(\infty)$, as follows.

\begin{proposition}
  \label{prop:words}
Let $(\mu_j)_{j\in\NN}$ be a sequence of dominant weights which is {additively cofinal}, in the sense that
\begin{equation}
  \label{eq:cofinal}
  \forall \lambda \in \bP^+, ~ \exists n\in\NN \text{ such that } \lambda \leq \sum_{j=0}^n \mu_j.
\end{equation}
There is a one-to-one correspondence between the set of semi-infinite Cartan words in $\prod_{j\in\NN}\cB(\mu_j)$ and the semi-infinite crystal $\crystal(\infty)$, given by
\begin{equation}
  \label{eq:word-decomposition}
  \ldots x_3 x_2 x_1 x_0 \ \mapsto\  x=(b_\lambda)_{\lambda\in\bP^+},
\end{equation}
where
\begin{equation}
  b_\lambda = \sfR_\lambda(x_nx_{n-1}\ldots x_1x_0)
\end{equation}
for any $n$ large enough so that $\lambda \leq \sum_{j=0}^n \mu_j$.
\end{proposition}

\begin{proof}

Let $x = (b_\lambda)_{\lambda\in\bP^+}\in\crystal(\infty)$.  Let us write $\lambda_n = \sum_{j=0}^n\mu_j$. Then for any $n$, we can realise $b_{\lambda_n}$ as a finite Cartan word,
\[
  b_{\lambda_n} = x_n\ldots x_1 x_0 \qquad \text{with }x_j \in \cB(\mu_j),
\]
where according to Equation~\eqref{eq:letter},
\[
  x_k = \sfR_{\mu_k}\circ\sfS_{\sum_{j=0}^{k-1}\mu_j}(b_{\lambda_n}).
\]
The compatibility relations \eqref{eq:RR}--\eqref{eq:SR} ensure that this  $x_k$ is independent of the word length $n$.  It is then straightforward to check that the map $x\mapsto \ldots x_3x_2x_1x_0$ is inverse to~\eqref{eq:word-decomposition}.
\end{proof}

\subsection{A Cartan braiding criterion for Cartan words}
\label{sec:braiding}

We will take a brief digression to expand upon what it means for a semi-infinite word $x=\ldots x_2 x_1 x_0$ to be a Cartan word.
We begin by recalling a useful technical tool from \cite{MatYun:crystals}.  For an antecedent to this, see \cite{EtiSchSol}.

\begin{definition}(\cite[Definition~5.4]{MatYun:crystals})
  \label{def:Cartan_braiding}
  Let $\mu,\nu\in\bP^+$, and let us write
  \[
    \iota_{\mu,\nu}:\cB(\mu+\nu) \to \cB(\mu)\otimes\cB(\nu)
  \]
  for the inclusion of the Cartan component.  The \emph{Cartan braiding} is the map
  \begin{equation*}
    \sigma = \sigma_{\cB(\mu),\cB(\nu)}:(\cB(\mu) \otimes \cB(\nu))\sqcup\{0\}  \to (\cB(\nu) \otimes \cB(\mu))\sqcup\{0\}
  \end{equation*}
  defined by
  \begin{equation*}
    \sigma: a\otimes b  \mapsto
    \begin{cases}
      \iota_{\nu,\mu}\circ\iota_{\mu,\nu}^{-1}(a\otimes b) &
        \text{if $a\otimes b$ is in the Cartan component,} \\
      0 & \text{otherwise.}\\

    \end{cases}
  \end{equation*}

  If $\mu_n,\ldots,\mu_0 \in \bP^+$, we write
  \begin{multline*}
    \sigma_j : \cB(\mu_n) \otimes \cdots \otimes \cB(\mu_{j}) \otimes \cB(\mu_{j-1}) \otimes \cdots \otimes \cB(\mu_0) \\
    \to
    \cB(\mu_n) \otimes \cdots \otimes \cB(\mu_{j-1}) \otimes \cB(\mu_{j}) \otimes \cdots \otimes \cB(\mu_0) \; \sqcup\{0\}
  \end{multline*}
  for the action
  \[
  \sigma_j = \id_{\mu_n} \otimes \cdots \otimes\id_{\mu_{j+1}} \otimes \sigma \otimes \id_{\mu_{j-2}} \otimes \cdots \otimes\id_{\mu_0}
  \]
  of the Cartan braiding on the $j$th and $(j\!-\!1)$th tensor legs.
\end{definition}

The utility of the Cartan braiding comes from the following property from \cite[Proposition~ 6.7]{MatYun:crystals}.  Our ordering convention here is reversed compared to \cite{MatYun:crystals}, since we are enumerating tensor products from right to left, but the proof is the same.

\begin{proposition}
  \label{prop:Cartan_test}
  A tensor product $b_n\otimes\cdots\otimes b_0$ is in the Cartan component if and only if, for all $k=1,\ldots,n$, we have
  \[
    \sigma_n\circ\sigma_{n-1}\circ\cdots\circ\sigma_k(b_n\otimes\cdots\otimes b_0) \neq 0.
  \]
\end{proposition}

This simplifies dramatically if the weights $\mu_j$ are all equal, because then the Cartan braiding on $\cB(\mu)\otimes\cB(\mu)$ is just the projection onto the Cartan component:
\begin{equation}
\label{eq:BxB_braiding}
  \sigma(b\otimes c) =
  \begin{cases}
    b\otimes c & \text{if $b\otimes c$ is in the Cartan component},\\
    0 & \text{otherwise}.
  \end{cases}
\end{equation}
Consequently, Proposition~\ref{prop:Cartan_test} implies that a word
\[
  x_n\ldots x_0 \in \prod_{j=0}^n \cB(\mu)
\]
is a Cartan word if and only if each successive pair $x_j\otimes x_{j-1}$ is in the Cartan component.

It follows that the same is true for semi-infinite Cartan words.  As a result, we obtain the following.   Recall that a weight $\mu\in\bP^+$ is \emph{strictly dominant} if
$\mu \geq \varpi_i$ for every fundamental weight $\varpi_i$.

\begin{corollary}
  \label{cor:subshift}
  Let $\mu\in\bP^{+}$ be strictly dominant.  There is a one-to-one correspondence
  \begin{multline}
  \label{eq:subshift}
    \crystal(\infty) \cong \{ \ldots x_2 x_1 x_0 \in \prod_{j=0}^\infty \cB(\mu) \mid \text{ $x_{j+1}\otimes x_{j}$ is in the Cartan}\\
    \text{component of $\cB(\mu) \otimes\cB(\mu)$ for every $j\in\NN$ } \}
  \end{multline}
\end{corollary}

This shows that the semi-infinite crystal is really a higher-rank analogue of a subshift of finite type.
We will see later (Section~\ref{sec:orbits}) that all semi-infinite Cartan words in $\prod_{j\in\NN}\cB(\mu)$ are eventually constant.

\subsection{Example: The crystal groupoid of \texorpdfstring{$SU(2)$}{SU(2)}}
\label{sec:SU2}

Consider the case of $K=SU(2)$.
The fundamental weight $\varpi=1$ is strictly dominant, so by Corollary~\ref{cor:subshift}, $\crystal(\infty)$ is the set of semi-infinite Cartan words in $\cB(\varpi)^\infty$.

Let us denote the elements of the fundamental crystal as $a$ and $b$, with $a$ being the highest-weight element.  The Cartan component of the tensor product $\cB(\varpi)\otimes\cB(\varpi)$ is
\[
  \{ a\otimes a, b\otimes a, b\otimes b \};
\]
this follows from Kashiwara's tensor product rule, see \emph{e.g.} \cite[Theorem~4.4.1]{HonKan}.  Therefore, a semi-infinite word in $\cB(\varpi)^\infty$ is a Cartan word if and only if it never contains the subword $ab$.  We obtain
\begin{equation}
  \label{eq:B_SU2}
  \crystal(\infty) = \{ b^\infty a^n \mid n \in \NN\} \sqcup \{a^\infty\} .
\end{equation}
We therefore have a bijection
\begin{equation}
  \label{eq:N_infty}
  \crystal(\infty) \stackrel{\cong}{\longrightarrow} \NN_\infty := \NN \sqcup \{\infty\},
\end{equation}
given by counting the number of occurrences of $a$.  Moreover, the bijection~\eqref{eq:N_infty} is a homeomorphism, where $\NN_\infty$ is endowed with its usual topology as the one-point compactification of the discrete set $\NN$.

Let us identify $\bP \cong \ZZ$ and $\bP^+ \cong \NN$, via $k\varpi\mapsto k$.  The shift action of $\bP^+$ on $\crystal(\infty)$ is
\begin{equation}
  \label{eq:S_SU2}
  \sfS_k : n\mapsto \max(n-k,0),
\end{equation}
with the convention that $\infty-k=\infty$. For any natural numbers $m'\geq m$ and $n'\geq n$, we have 
\[
\sfS_{m'} (b^\infty a^m) = \sfS_{n'} (b^\infty a^n),
\]
and so $(m,m'-n',n) \in \cG(SU(2))$. Hence, the crystal groupoid is
\begin{align}
  \label{eq:G_SU2}
  \cG(\SU(2)) &= \{(m,k,n) \in \NN_\infty\times\ZZ\times\NN_\infty \mid
  \nonumber \\
  &\hspace{1cm}\text{ $m$ and $n$ are either both finite or both infinite}\}
  \nonumber
  \\
    &\cong (\NN^2 \sqcup \{\infty\}^2) \times \ZZ
    \rightrightarrows \NN_\infty,
\end{align}
where $\NN^2\rightrightarrows\NN$ and $\{\infty\}^2\rightrightarrows\{\infty\}$ are pair groupoids (the latter trivial) and $\ZZ$ is the additive group.  The topology on $\cG(\SU(2))$ is determined by the following facts:
\begin{itemize}
  \item under the inherited topologies, the subgroupoids $\NN^2\times\ZZ$ and\linebreak $\{\infty\}^2\times \ZZ$ are discrete, and $\{\infty\}^2\times \ZZ$ is closed,
  \item a sequence in $\NN^2\times\ZZ$ converges to $(\infty,k,\infty)$ if and only if it is eventually of the form $(n_i+k,k,n_i)$ with $n_i\to\infty$.
\end{itemize}

From this description, we see that the crystal groupoid of $\SU(2)$ has precisely two orbits, with representatives $b^\infty$ and $a^\infty$ respectively, and isotropy groups $\ZZ$ at every point.  Moreover, the orbit of $a^\infty$ lies in the closure of the orbit of $b^\infty$.  We will return to this structure in Section~\ref{sec:structure}.


\section{The crystal algebra}
\label{sec:algebra}

\subsection{The crystal groupoid algebra \texorpdfstring{$\Cc^\infty(\GK)$}{}}

The following definitions were made in the introduction.

\begin{definition}
    We write $\Cc^\infty(\GK)$ for the $*$-algebra of compactly supported locally constant functions on $\GK$ equipped with the convolution product and the involution
    \[
        f^*(g) = \overline{f(g^{-1})}, \qquad g=(x,\lambda,y)\in\GK.
    \]
    The $C^*$-completion is denoted $C^*(\GK)$.
\end{definition}

There is a canonical gauge action of the torus $T=\widehat{\bP}$ on the convolution algebra $\Cc^\infty(\GK)$ of the crystal groupoid, which we now recall.  A gauge action of $T$ on a $*$-algebra $\cA$ can be determined by a $\bP$-grading of $\cA$, \emph{i.e.} a decomposition
\begin{equation}
\label{eq:grading}
  \cA = \bigoplus_{\lambda\in\bP} \cA_\lambda,
\end{equation}
with
\[
  \cA_\lambda \cA_\mu \subseteq \cA_{\lambda+\mu}, \qquad \cA_\lambda^* \subseteq \cA_{-\lambda}.
\]
Specifically, writing $(z,\lambda)\mapsto z^\lambda$ for the canonical pairing $T \times \bP \to \TT$, the corresponding action satisfies $z \cdot a = z^\lambda a$ for $a \in \cA_\lambda$. 

In the case of the crystal groupoid algebra, the map $c : (x, \lambda, y) \mapsto \lambda$ is a continuous $\bP$-valued cocycle on $\GK$, and the $\lambda$-isotypical component of $\Cc^\infty(\GK)$ is
\begin{equation}
  \label{eq:CcG_gauge_action}
  \Cc^\infty(\GK)_\lambda = \{ f\in \Cc^\infty(\GK) \mid \supp(f) \subseteq c^{-1}(\lambda)\}.
\end{equation}
Throughout this article, we will focus on this $\bP$-grading rather than the gauge action, as is done for instance in \cite[\S{}4]{Aranda:Kumjian-Pask}.

\subsection{The crystallised function algebra \texorpdfstring{$\cO(K_0)$}{}}

The $*$-algebra $\cO(K)$ of regular functions on $K$ is spanned by the matrix coefficients of irreducible representations.  For a representation of highest weight $\mu\in\bP^+$, we will use the notation
\begin{equation}
  \label{eq:matrix_coeff}
  \ip{\xi|\cdot|\eta} : g \mapsto \xi (\pi_\mu(g)\eta), \qquad \eta\in V(\mu),\ \xi \in V(\mu)^*
\end{equation}
for these matrix coefficients.  The algebra $\cO(K)$ also has a $\bP$-grading, whose $\lambda$-isotypical component is
\begin{equation}
  \label{eq:OKq_gauge_action}
  \cO(K)_\lambda = \Vect \{ \ip{\xi|\cdot|\eta} \mid \text{ $\eta$ is a vector of weight $\lambda$}  \}.
\end{equation}

The $*$-algebra $\cO(K)$ admits a quantisation $\cO(K_q)$ to a family of compact quantum group algebras for $0<q<\infty$ with the same underlying vector space and the same $\bP$-grading.  In \cite{MatYun:crystals} it was shown that this family of $*$-algebras can be extended to $q=0$ and $q=\infty$ in such a way that their $C^*$-closures $C(K_q)$ form a continuous field of $C^*$-algebras.
Moreover, the limit algebras $\cO(K_0)\cong\cO(K_\infty)$ and $C(K_0)\cong C(K_\infty)$ are $\bP$-graded isomorphic to a Kumjian-Pask algebra and a higher-rank graph $C^*$-algebra, respectively; see \cite[Theorem~9.4]{MatYun:crystals}.

Since higher-rank graph $C^*$-algebras are isomorphic to the groupoid $C^*$-algebras of their infinite-path groupoids \cite{KumPas}, we immediately conclude that $C(K_0)$ is isomorphic to a groupoid $C^*$-algebra.  However, the higher-rank graph associated to $C(K_0)$ is somewhat complicated to describe, see for instance the diagrams at the end of \cite{MatYun:crystals} for type $C_2$.  This makes computing its path groupoid difficult.
The main theorem of this article is that the crystal groupoid $\GK$ defined above is the appropriate groupoid model.

The key to proving this is to determine a set of generators and relations for the crystal groupoid algebra.  If $a\in\cB(\mu)$ and $b\in\cB(\nu)$ are finite crystal elements, we will use the notation
\begin{equation}
  \label{eq:indicator}
  \mathbf{1}_{Z(a,b)} \in \Cc^\infty(\GK)
\end{equation}
to denote the indicator function of the basis clopen set $Z(a,b)$ from Equation~\eqref{eq:Z}.  Recall that $\emptyset\in\cB(0)$ denotes the unique element of the trivial crystal.

Given $\mu,\nu \in \bP^+$ and given $a \in \cB(\mu)$ and $b \in \cB(\nu)$, we write
\begin{align*}
    \CCE(a,b) = \{c \in \cB(\mu + \nu) :{}& c \text{ is a Cartan word and }c = a'a = b'b\\
     &\qquad\text{for some }a' \in \cB(\nu)\text{ and }b' \in \cB(\mu)\}
\end{align*}
for the set of \emph{common Cartan extensions} of $a$ and $b$ in the $(\mu\!+\!\nu)$-crystal.

\begin{lemma}
  \label{lem:relations}
Let $a\in\cB(\mu)$, $b\in\cB(\nu)$.  We have:
\begin{enumerate}
    \item[(a)]
    $\mathbf{1}_{Z(\emptyset,\emptyset)}=1$,
    \item[(b)]
    $\mathbf{1}_{Z(a,b)}^* = \mathbf{1}_{Z(b,a)}$,
    \item[(c)]
    $\mathbf{1}_{Z(\emptyset,a)}\mathbf{1}_{Z(\emptyset,b)} = \mathbf{1}_{Z(\emptyset,ab)}$ if $ab$ is a Cartan word, and $0$ otherwise.
    \item[(d)]
    $\mathbf{1}_{Z(a,\emptyset)}\mathbf{1}_{Z(\emptyset,b)} = \mathbf{1}_{Z(a,b)}$,
    \item[(e)]
    \(
      \ds\mathbf{1}_{Z(\emptyset,a)}\mathbf{1}_{Z(b,\emptyset)} = \!\!\! \sum_{a'a=b'b \in \CCE(a,b)} \!\!\!\!\!\!\mathbf{1}_{Z(a',b')}.
    \)
\end{enumerate}
Thus, $\bigcup_{\mu \in \bP^+} \big\{\mathbf{1}_{Z(\emptyset, a)} \mid a \in \cB(\mu)\big\}$ generates $\Cc^\infty(\GK)$ as a $*$-algebra.
\end{lemma}

\begin{proof}
These are elementary calculations.  For (a), note that $Z(\emptyset,\emptyset)$ is the set of units of $\GK$.  Part (b) is direct from the definitions.  For (c), note that the composable pairs in $Z(\emptyset,a) \times Z(\emptyset,b)$ are precisely those of the form $((z,-\mu,za),(za,-\nu,zab))$ with $z\in\crystal(\infty)$. The proof of (d) is similar.

For (e), we observe that elements in $Z(\emptyset,a)$ and $Z(b,\emptyset)$ can be written, respectively, in the form
\begin{align*}
  &(za',-\mu,za'a), && \text{for some $z\in\crystal(\infty)$, $a'\in\cB(\nu)$,}\\
  &(z'b'b,\nu,z'b'), && \text{for some $z'\in\crystal(\infty)$, $b'\in\cB(\mu)$.}
\end{align*}
These elements are composable in $\GK$ if and only if $a'a=b'b$ and $z=z'$.  In this case, the composition is
\begin{equation}
\label{eq:decomposition}
(za',-\mu,za'a)(za'a,\nu,zb') = (za',\nu-\mu,zb').
\end{equation}
To complete the calculation, we need to check that every element $(za',\nu-\mu,zb')\in Z(a',b')$ admits a unique decomposition as a product in $Z(\emptyset,a)\times Z(b,\emptyset)$.  Certainly, if such a decomposition exists it is unique, since it must have the form~\eqref{eq:decomposition}.  Existence amounts to proving that $za'a$ is a Cartan word, \emph{i.e} that $z_\lambda a'a$ is a Cartan word for any finite crystal element $z_\lambda\in\cB(\lambda)$ such that $z_\lambda a'$ and $z_\lambda b'$ are both Cartan words.  For this, we use the Cartan braiding criterion of Proposition~\ref{prop:Cartan_test}.  Our hypothesis is equivalent to $\sigma(z_\lambda\otimes a')\neq0$ and $\sigma(z_\lambda\otimes b')\neq 0$.  But then
\[
  \sigma_2(z_\lambda\otimes a' \otimes a) \neq0,\quad\text{and}\quad
  \sigma_2\circ\sigma_1(z_\lambda\otimes a'\otimes a) = \sigma_2(z_\lambda\otimes b'\otimes b) \neq 0,
\]
from which we see that $z_\lambda\otimes a'\otimes a$ is in the Cartan component.  The result follows.
\end{proof}

On the function algebra side, let us introduce the notation
\begin{equation}
  \label{eq:highest_weight_vector}
  b_e(\mu) \in \cB(\mu)
\end{equation}
for the highest-weight element of the irreducible crystal $\cB(\mu)$.  This notation is chosen because we will later use
\begin{equation}
  \label{eq:extremal_weight_vector}
  b_w(\mu) \in \cB(\mu)
\end{equation}
to denote the extremal element of weight $w\mu$ where $w\in W$ is any element of the Weyl group.

By \cite[\S{}4]{MatYun:crystals}, the crystallised function algebra $\cO(K_0)$ is generated as a $*$-algebra by the matrix coefficients of the form
\begin{equation}
  \label{eq:fa}
  f_a := \ip{a\mid\cdot\mid b_e(\mu)} \quad \in\cO(K_0),
\end{equation}
with $a\in\cB(\mu)$.
To be more accurate, the notation~\eqref{eq:fa} refers to the $q=0$ limit of the matrix coefficients in $\cO(K_q)$ corresponding to any vectors in the crystal basis for $V(\mu)$ which specialise to $a$ and $b_e(\mu)$ at $q=0$.  We refer to \cite[\S{}4--5]{MatYun:crystals} for the details.

\begin{theorem}
\label{thm:isomorphism}
There is a $\bP$-graded $*$-isomorphism
    $\cO(K_0)\cong \Cc^\infty(\GK)$ and $C(K_0)\cong C^*(\GK)$, where $\cO(K_0)$ and $C(K_0)$ are equipped with the opposite $\bP$-grading to~\eqref{eq:OKq_gauge_action}.  The isomorphism is given on generators by
    \begin{equation}
    \label{eq:isomorphism}
      f_a  \mapsto \mathbf{1}_{Z(\emptyset,a)}
    \end{equation}
    for every $a\in\cB(\mu)$.
\end{theorem}

\begin{proof}
    By \cite[\S{}8-9]{MatYun:crystals}, $\cO(K_0)$ is the universal $*$-algebra with generators $f_a$, with $a$ varying over all finite crystals, and the following relations:
    \begin{enumerate}
        \item[(1)]
        for any $a\in\cB(\mu)$, $b\in\cB(\nu)$,
        $f_af_b = f_{ab}$ if $a\otimes b$ is in the Cartan component, and $f_af_b=0$ otherwise;
        \item[(2)]
        for any $a\in\cB(\mu)$, $b\in\cB(\nu)$,
        $f_af_b^* = \!\!\!\ds\sum_{a'a=b'b \in \CCE(a,b)} \!\!\!\!\!\! f_{a'}^*f_{b'}$;
        \item[(3)]
        for any $\mu\in\bP^+$,
        $\ds\sum_{a\in\cB(\mu)} f_a^*f_a = 1.$
    \end{enumerate}
    It is straightforward to check, using Lemma~\ref{lem:relations}, that the generators $\mathbf{1}_{Z(\emptyset,a)}$ of $\Cc^\infty(\GK)$ satisfy the same relations, so the formula~\eqref{eq:isomorphism} does define a $*$-morphism $\varphi_0 : \cO(K_0) \to \Cc^\infty(\GK)$.    Moreover, if $a\in\cB(\mu)$ then the generators $f_a\in\cO(K_0)$ and $\mathbf{1}_{Z(\emptyset,a)}\in\Cc^\infty(\GK)$ are of degree $\mu$ and $-\mu$ for their respective $\bP$-gradings, so $\varphi_0$ is grading-preserving, with the opposite $\bP$-grading action on $\cO(K_0)$.

    By \cite[Theorem~9.4]{MatYun:crystals}, the algebra $\cO(K_0)$ is graded isomorphic to the Kumjian-Pask algebra \cite{MatYun:crystals} of a row-finite higher-rank graph $\Lambda$ without sources, again with the opposite $\bP$-grading  on $\cO(K_0)$.  Therefore, once we show that the images under $\varphi_0$ of the vertex projections are all nonzero, it will follow from the graded uniqueness theorem \cite[Theorem~4.1]{Aranda:Kumjian-Pask} that the map~\eqref{eq:isomorphism} is an isomorphism.  By \cite[Corollary 8.15]{MatYun:crystals}, the vertex projections can be written as sums of projections of the form $p_a= f_a^*f_a$ with $a\in\cB(\mu)$ for any strictly dominant weight $\mu$.  But
    \[
      \varphi_0(p_a) = \varphi_0(f_a^*f_a) = \mathbf{1}_{Z(a,a)},
    \]
    and $Z(a,a)$ is non-empty because it contains $(b_{w_0}(\infty)a,0,b_{w_0}(\infty)a)$ where $w_0$ is the longest Weyl group element, see \cite[Lemma 5.2]{MatYun:crystals}.  Therefore, the images of the vertex projections are nonzero.

    This completes the proof for $\cO(K_0)$. By \cite[Theorem~9.4]{MatYun:crystals} again, the aforementioned graded isomorphism $\cO(K_0) \cong KP_\CC(\Lambda)$ extends to an isomorphism $C(K_0) \cong C^*(\Lambda)$. The universal property of $C^*(\Lambda)$ therefore ensures that $\varphi_0$ extends to a homomorphism $\varphi : C(K_0) \to C^*(\GK)$. This $\varphi$ intertwines the gauge actions because $\varphi_0$ is graded, and we already saw that each $\varphi(p_a) = \varphi_0(p_a) \not= 0$. So the Gauge-Invariant Uniqueness Theorem \cite[Theorem~3.4]{KumPas} shows that $\varphi$ is injective. It is surjective because its range contains the dense subalgebra $\Cc^\infty(\GK)$ of $C^*(\GK)$.
\end{proof}

\begin{example}
From the description of the crystal groupoid $\cG(SU(2))$ in Equation~\eqref{eq:G_SU2}, we get the short exact sequence
\[
  0 \longrightarrow M_\infty(\CC)\otimes\Cc(\ZZ) \longrightarrow \cO(\SU_0(2))
    \longrightarrow \Cc(\ZZ) \longrightarrow 0,
\]
where $M_\infty(\CC) = \underset{n\in\NN}{\underrightarrow{\lim}}\,M_n(\CC)$ and $\Cc(\ZZ) \cong \cO(\TT)$.
\end{example}

\section{Structure of the crystal groupoid}
\label{sec:structure}

In this section, we examine the topological and combinatorial structure of the semi-infinite crystal $\crystal(\infty)$ and the crystal groupoid $\GK$.

\subsection{Order structure of Cartan words}
\label{sec:Cartan_order}

For weights $\omega_1,\omega_2 \in \bP$, let us write $\omega_1\preceq\omega_2$ if $\omega_2-\omega_1 \in \bQ^+ = \NN\Simples$, and $\omega_1\prec\omega_2$ if in addition $\omega_1\neq\omega_2$.

\begin{lemma}
\label{lem:braiding_order}
  Let $a,a'\in\cB(\lambda)$, $b,b'\in\cB(\mu)$ with $\sigma(a\otimes b) = b'\otimes a'$ nonzero.  Then either (i) $a=a'$ and $b=b'$ or (ii) $\wt(a') \succ \wt(a)$ and $\wt(b') \prec \wt(b)$.
\end{lemma}

\begin{proof}
  By Theorem~5.3 and Definition~5.4 of \cite{MatYun:crystals}, the Cartan braiding $\cB(\lambda)\otimes\cB(\mu) \to \cB(\mu)\otimes\cB(\lambda)$ is given by the $q=0$ limit of a suitable multiple of the braiding map $V(\lambda)\otimes V(\mu) \to V(\mu)\otimes V(\lambda)$.  The result therefore follows from the ordering property of the braiding map, see Equation (5.1) in \cite{MatYun:crystals}.
\end{proof}

\begin{corollary}
  \label{cor:BxB_decreasing}
  Let $a,b\in\cB(\lambda)$.  If $a\otimes b$ is in the Cartan component of $\cB(\lambda)\otimes\cB(\lambda)$, then either $a=b$ or $\wt(a) \prec \wt(b)$.  Moreover, $b\otimes b$ is in the Cartan component if and only if $b$ has extremal weight.
\end{corollary}

\begin{proof}
The first statement follows by applying the previous lemma to Equation~\eqref{eq:BxB_braiding}.
For the second statement, suppose first that $b\in\cB(\lambda)$ has extremal weight.  Then $b\otimes b$ has extremal weight in $\cB(\lambda)\otimes\cB(\lambda)$, so belongs to the Cartan component.  Conversely, Equation~\eqref{eq:BxB_braiding} shows that $b\otimes b$ is in the Cartan component if and only if $\sigma(b\otimes b) = b\otimes b$, so the result follows from \cite[Lemma~10.1]{MatYun:crystals}.
\end{proof}

It follows from Corollary \ref{cor:BxB_decreasing} that if $\lambda\in\bP^+$ is strictly dominant, then any semi-infinite Cartan word
\[
  \ldots x_3x_2x_1x_0 \in \prod_{j=0}^{\infty} \cB(\lambda)
\]
must eventually stabilise to an extremal element---that is, there is $w\in W$ such that $x_j = b_w(\lambda)$ for all $j$ sufficiently large, where $b_w(\lambda)$ is the extremal element from Equation \eqref{eq:extremal_weight_vector}.  

We can obtain more precise information about the structure of semi-infinite Cartan words by using the theory of Demazure crystals \cite{Littelmann,Kas:Demazure,AssDraGon}.  Let $w\in W$ be a Weyl group element, and let $w=s_{i_1}\ldots s_{i_k}$ be a reduced expression for $w$, where $s_i$ denotes the simple reflection in $\alpha_i$.  The Demazure crystal $\cB_{w}(\lambda)\subseteq\cB(\lambda)$ is the set
  \begin{equation}
  \label{eq:Demazure_crystal}
     \cB_w(\lambda) = \{ f_{i_1}^{m_1}f_{i_2}^{m_2}\ldots f_{i_k}^{m_k} (b_e(\lambda))  \in \cB(\lambda) \mid
     m_1,\ldots,m_k\in\NN\} .
  \end{equation}
As the notation suggests, this set is not dependent on the choice of reduced expression for $w$.  The utility of Demazure crystals for our purposes comes from the following result, due to Assaf, Dranowski and Gonz\'{a}lez.

\begin{proposition}[{\cite[Lemma 9.1]{AssDraGon}}]
  \label{prop:Demazure}
  Let $\lambda,\mu\in\bP^+$ and fix $w\in W$.  Write $\lambda=\sum_i\lambda_i\varpi_i$ and $\mu=\sum_i\mu_i\varpi_i$ with $\lambda_i,\mu_i\in\NN$, and suppose that  $\mu_i=0$ whenever $\lambda_i=0$.  If $b_w(\lambda)\otimes a$ is in the Cartan component of $\cB(\lambda)\otimes\cB(\mu)$, then $a\in\cB_w(\mu)$. 
\end{proposition}

Combining this with Proposition 4.5 from \cite{Matassa:some_results}, we can characterise which tensor products of extremal elements belong to the Cartan component.  This was independently observed by M.~Matassa and T.~Schulz (personal communication).
To state the result, we need some notation.  

Let $\lambda\in\bP^+$ and write
\begin{equation}
  \lambda=\sum_{i\in\Simples} \lambda_i \varpi_i
\end{equation}
as above.  We define the \emph{support} of $\lambda$ to be
\begin{equation}
\label{eq:support}
  I = \{ i\in\Simples \mid \lambda_i\neq 0\}.
\end{equation}
The stabiliser subgroup of $\lambda$ under the Weyl group action depends only on its support, and we denote it by $W_I$.  
We warn the reader that with this convention, $W_I$ is generated by the reflections in the simple roots \emph{not} in $I$, so that it corresponds to $W_{S}$ with $S=I^c$ from elsewhere in the literature.
For $w\in W$, we write $[w]_I$ for the coset $wW_I \in W/W_I$.   Note that the extremal weight elements of $\cB(\lambda)$ satisfy
\[
  b_w(\lambda) = b_{w'}(\lambda) \quad \Leftrightarrow \quad [w]_I = [w']_I.
\]

The $W_I$-cosets are ordered by the Bruhat order, see \emph{e.g.}\ \cite[Ch.2]{BjoBre}.  The following fact is well-known, see \emph{e.g.}\ the proof of \cite[Proposition 3.4]{AssDraGon}.  We include a proof for completeness.

\begin{proposition}
  \label{prop:extremal_in_Demazure}
  Let $\lambda\in\bP^+$ with support $I$, and let $v,w\in W$.  Then $b_v(\lambda) \in \cB_w(\lambda)$ if and only if $[v]_I\leq[w]_I$ for the Bruhat order.
\end{proposition}

\begin{proof}
  If $[v]_I\leq[w]_I$ then it follows from the definitions that $\cB_v(\lambda) \subseteq \cB_w(\lambda)$ and hence $b_v(\lambda)\in\cB_w(\lambda)$.

  For the converse, we work by induction on the length of a minimal coset representative of $[w]_I$.  Indeed, without loss of generality, we may suppose that $w$ is the minimal length element in its coset.
  If $\ell(w)=0$ then $\cB_w(\lambda)=\{b_e(\lambda)\}$ so the result is trivial.  Now suppose $\ell(w)>0$ and assume we have proven the result for all $w'\in W$ with coset representative of length less than $\ell(w)$.  Let $s_i$ be a simple reflection such that $\ell(s_iw) = \ell(w) - 1$, meaning that $s_i$ is the first simple reflection in a reduced expression for $w$.  Since $b_v(\lambda)\in\cB_{w}(\lambda)$ we have
  \begin{equation}
    \label{eq:Demazure_induction}
    b_v(\lambda) = f_i^m b \qquad \text{for some }b\in\cB_{s_iw}(\lambda).
  \end{equation}
  Since $b_v(\lambda)$ is extremal, it is either at the top or the bottom of its $\alpha_i$-string.  If it is at the top then $m=0$ in Equation~\eqref{eq:Demazure_induction}, which means $b_v(\lambda)\in\cB_{s_iw}(\lambda)$ and by induction we get $[v]_I\leq [s_iw]_I < [w]_I$.  If it is at the bottom, then $b_{s_iv}(\lambda) = e_i^M (b_{v}(\lambda))$, where $M\in\NN$ is the largest integer such that $e_i^M(b_v(\lambda))\neq0$.  From this and~\eqref{eq:Demazure_induction} we get
  \[
    b_{s_iv}(\lambda) = e_i^{M-m} b \in \cB_{s_iw}(\lambda),
  \]
  since Demazure crystals are stable under $e_i$.  By induction, we get $[s_iv]_I\leq [s_iw]_I$, \emph{i.e.} $s_iv' \leq s_iw$, where $s_iv'$ is the minimal length coset representative of $[s_iv]_I$.  
  Since $\ell(w) = \ell(s_iw)+1$, this implies $v' \leq w$. Since $[v']_I = [v]_I$, this completes the proof.
\end{proof}

We can now return to Cartan words.

\begin{proposition}
  \label{prop:extremal_products}
  Let $\lambda\in\bP^+$ with support $I$, and let $v,w\in W$.
  Then $b_w(\lambda)\otimes b_v(\lambda)$ is in the Cartan component of $\cB(\lambda)\otimes\cB(\lambda)$ if and only if $[v]_I \leq [w]_I$ for the Bruhat order.
\end{proposition}

\begin{proof}
  Let us write $w^I$ for the minimal-length element of $[w]_I \in W/W_I$.  Then $[v]_I \leq [w]_I$ if and only if $v^I\leq w^I$.  In this case, \cite[Proposition~4.5]{Matassa:some_results} implies that $b_w(\lambda)\otimes b_v(\lambda) = b_{w^I}(\lambda) \otimes b_{v^I}(\lambda)$ is in the Cartan component.
  Conversely, if $b_w(\lambda)\otimes b_v(\lambda)$ is in the Cartan component, then Proposition \ref{prop:Demazure} and Proposition \ref{prop:extremal_in_Demazure} imply that $[v]_I\leq[w]_I$.
\end{proof}

\subsection{Orbits and isotropy groups}
\label{sec:orbits}

Two semi-infinite crystals $x,y \in \crystal(\infty)$ are \emph{tail equivalent} if there exist $\mu,\nu \in \bP^+$ such that $\sfS_\mu(x) = \sfS_\nu(y)$.
The orbits of the crystal groupoid are precisely the tail equivalence classes of $\crystal(\infty)$.  As observed in the previous section, the tail equivalence classes are indexed by the Weyl group.  Let us be more specific.

\begin{lemma}
\label{lem:extremal_words}
Let $\lambda\in\bP^{+}$ be a strictly dominant weight and let $x\in\cB(\lambda)$.  The semi-infinite word
\[
  x^\infty = \ldots xxx \ \in \cB(\lambda)^\infty
\]
is a Cartan word if and only if $x$ has extremal weight, \emph{i.e.}, $x=b_w(\lambda)$ for some $w\in W$.  Moreover, for any additively cofinal sequence $(\mu_j)_{j\in\NN}$ of dominant weights, we have
\begin{equation}
  \label{eq:b_w_infty}
   b_w(\lambda)^\infty = \ldots b_w(\mu_3) b_w({\mu_2}) b_w({\mu_1}) b_w({\mu_0})
\end{equation}
\end{lemma}

\begin{proof}
  The first statement follows from Corollary~\ref{cor:subshift} and Corollary~\ref{cor:BxB_decreasing}.  The semi-infinite words on both sides of~\eqref{eq:b_w_infty} have the property that all finite truncations are extremal elements, so they are equal.
\end{proof}

\begin{definition}
  \label{eq:b_w}
We will denote the semi-infinite crystal element~\eqref{eq:b_w_infty} by
\begin{equation*}
  b_w(\infty) \in \crystal(\infty),
\end{equation*}
and refer to it as the \emph{extremal semi-infinite crystal} associated to $w\in W$.
\end{definition}

Corollary \ref{cor:BxB_decreasing} implies the following.

\begin{theorem}
\label{thm:orbits}
The set
\(
  \{b_w(\infty) \mid w\in W\},
\)
is a complete set of orbit representatives of $\GK$.
\end{theorem}

\begin{remark}
  \label{rmk:shift-invariant}
  Note that a semi-infinite crystal $x$ belongs to the orbit $[b_w(\infty)]$ if and only if $\sfS_\mu(x) = b_w(\infty)$ for some sufficiently large $\mu\in\bP^+$.  The $b_w(\infty)$ are precisely the shift-invariant elements of $\crystal(\infty)$.
\end{remark}

\begin{corollary}
  \label{cor:isotropy}
The isotropy group of $\GK$ at any point is isomorphic to $\bP$.\qed
\end{corollary}

\subsection{Topology of the orbit space}
\label{sec:orbit_topology}

We will write $\crystal(\infty)/\GK$ for the set of orbits of the crystal groupoid, equipped with the quotient topology from $\crystal(\infty)$.
Given $x\in\crystal(\infty)$ we will write
\begin{equation}
\label{eq:orbit_notation}
  [x] \in \crystal(\infty)/\GK
\end{equation}
for its $\GK$-orbit.

\begin{theorem}
\label{thm:orbit_topology}
  Under the bijection
  \[
    W \to \crystal(\infty) / \GK; \qquad w\mapsto[b_w(\infty)]
  \]
  of Theorem~\ref{thm:orbits}, the topology on $\crystal(\infty)/\GK$ is determined by the fact that, for any $v,w\in W$,
  \begin{equation}
  \label{eq:orbit_topology}
    [b_v(\infty)] \in \overline{[b_w(\infty)]}
    \quad \Leftrightarrow \quad
    v \leq w
\end{equation}
\end{theorem}

\begin{proof}
  Firstly, suppose that $v \leq w$.  Fix a strictly dominant weight $\lambda\in\bP^+$.
  By Proposition \ref{prop:extremal_products},
  \[
    z_n := b_w(\lambda)^\infty b_v(\lambda)^n
  \]
  is a semi-infinite Cartan word for every $n$.  These $z_n$ belong to the orbit $[b_w(\infty)]$ and converge to $b_v(\infty)$.

  Conversely, suppose that $[b_v(\infty)] \in \overline{[b_w(\infty)]}$, so there is a sequence $(z_n)$ in $[b_w(\infty)]$ which converges to an element $z$ of $[b_v(\infty)]$.  By applying the shift $\sfS_\mu$ to this sequence, for $\mu\in\bP^+$ sufficiently large, we may assume that $z=b_v(\infty)$, see Remark~\ref{rmk:shift-invariant}.  Then for $n$ sufficiently large, the expression of $z_n$ as a semi-infinite Cartan word in $\cB(\lambda)^\infty$ has the form
  \[
    z_n = b_w(\lambda)^\infty x b_v(\lambda)
  \]
  for some finite Cartan word $x\in\cB(\lambda)^k$. 
  By Proposition \ref{prop:Demazure}, the image of $x\otimes b_v(\lambda)$ in $\cB((k+1)\lambda)$ lives in the Demazure crystal $\cB_w((k+1)\lambda)$, and from the definition of the Demazure crystal it follows that $b_v(\lambda)\in\cB_w(\lambda)$.  Proposition \ref{prop:extremal_in_Demazure} then gives $v\leq w$. 
\end{proof}


\section{The crystal groupoid of a flag variety}
\label{sec:homogeneous spaces}

\subsection{Quantised function algebras of flag varieties}

The crystallisation of the coordinate algebra of a flag variety was described in \cite{MatYun:crystals} as a $*$-subalgebra of a higher-rank graph $C^*$-algebra.  We are using ``flag variety'' here as a stand-in term for the range of Poisson homogeneous spaces considered, for instance, in \cite{NesTus:functions}.  These include the flag varieties of a compact semisimple Lie group, their canonical torus bundles, as well as a family of intermediate homogeneous spaces.

In this section, we will largely follow the notation of \cite{NesTus:functions}, with the notable exception that we will index flag varieties not by the set of simple roots $S\subseteq\Simples$ of the stabiliser subgroup, but rather with the complement of this set $I=S^\rmc$.  We begin by recalling the definitions.

Let $K$ be a simply connected compact semisimple Lie group.  Let $I = S^\rmc \subseteq \Simples$ be a set of simple roots of $K$.  Let $K^I$ be the closed subgroup of $K$ whose complexified Lie algebra is
\begin{equation}
\label{eq:k_I}
  \lie{k}^I_{\CC} = \lie{h} \oplus \!\!\bigoplus_{\alpha\in\Phi\cap\Vect(I^\rmc)}\hspace{-2ex} \lie{g}_{\alpha},
\end{equation}
where $\lie{g}=\lie{k}_\CC$, $\lie{h}=\lie{t}_\CC$, and $\Phi$ is the set of roots of $\lie{g}$.  Let $\widetilde{K}^I$ denote the semisimple component of $K^I$.  As mentioned, these groups are denoted in \cite{NesTus:functions} as $K^S$ and $\widetilde{K}^S$, respectively, with $S=I^\rmc$.

The homogeneous space
\[
X^I = K/K^I
\]
is a generalised flag variety and
\[
Y^I = K/\widetilde{K}^I \to X^I
\]
is the canonical compact torus bundle over $X^I$.
Let $\bPI$ denote the sublattice of $\bP$ consisting of weights orthogonal to the roots of $\widetilde{K}^I$, \emph{i.e.},
\begin{align}
  \label{eq:bP_I}
  \bPI &= \bP \cap \Vect\{\varpi_i \mid \alpha_i \in  I\},
\end{align}
and put $\bPIplus = \bPI\cap\bP^+$.
This agrees with the notation of \cite{MatYun:crystals}, but $\bPI$ corresponds to $\bP(S)$ in \cite{NesTus:functions}.

For any sublattice
\[
  L \leq \bPI,
\]
we have a homogeneous space intermediate between $K/\widetilde{K}^I$ and $K/K^I$ as follows.  Let $T_L$ be the annihilator of $L$ in $T$ and let $K^{I,L}$ be the subgroup of $K$ generated by $\widetilde{K}^I$ and $T_L$. We put
\begin{equation}
  \label{eq:XIL}
  X^{I,L} = K/K^{I,L}.
\end{equation}
The extreme cases are
\begin{equation}
  \label{eq:XI}
  K^{I,\mathbf{0}} = K^I, \qquad X^{I,\mathbf{0}} = X^I
\end{equation}
and
\begin{equation}
  \label{eq:YI}
  K^{I,\bPI} = \widetilde{K}^I, \qquad X^{I,\bPI} = Y^I.
\end{equation}
Note also that the case where $I=\Simples$ and $L\leq\bP$ is a finite index subgroup includes the example of any non--simply connected quotient of $K$.

The $*$-algebras of regular functions $\cO(X^{I,L}_q)$ and continuous functions $C(X^{I,L}_q)$ on the quantised homogeneous spaces $X^{I,L}_q$ with $0<q<\infty$ have been studied by many authors.  We mention in particular \cite{StoDij,Stokman:quantum_orbit_method,NesTus:functions}.  The $q=0$ limits $\cO(X^{I,L}_0)$ were introduced in \cite{MatYun:crystals}.  We recall the definition.

\begin{definition}
\label{def:OXIL}
We resume the notation of Equation~\eqref{eq:fa}. The crystallised function algebra $\cO(Y^I_0)$ is the $*$-subalgebra of $\cO(K_0)$ generated by the elements of the form
  \begin{equation}
  \label{eq:fa_for_YI}
    f_a = \ip{a\mid\cdot\mid b_e(\nu)}, \qquad \text{with }\nu\in\bPIplus,\ a\in\cB(\nu).
  \end{equation}
  The algebra $\cO(X^{I,L}_0)$ is the $*$-subalgebra of $\cO(Y^I_0)$ spanned by elements with graded degree in $L$.
\end{definition}

Our first step in obtaining a groupoid realisation of $\cO(X^{I,L}_0)$ is to use Theorem~\ref{thm:isomorphism} to identify $\cO(X_0^{I,L})$ as a $*$-subalgebra of the groupoid algebra $\Cc^\infty(\GK)$.

\begin{lemma}
\label{lem:flag_subalgebra}
  Under the isomorphism $\cO(K_0) \cong \Cc^\infty(\GK)$ of Theorem~\ref{thm:isomorphism}, we have
  \begin{equation}
    \label{eq:flag_subalgebra}
    \cO(Y^I_0) \cong \Vect\{\mathbf{1}_{Z(a,b)} \mid
      a\in\cB(\mu),\ b\in\cB(\nu)
      \text{ with } \mu,\nu\in\bPIplus\}.
  \end{equation}
  The algebra $\cO(X^{I,L}_0)$ is the subalgebra of $\cO(Y^I_0)$ spanned by elements with graded degree in $L$.
\end{lemma}

\begin{proof}
  The relations of Lemma~\ref{lem:relations} show that the right-hand side of~\eqref{eq:flag_subalgebra} is a $*$-subalgebra of $\cO(K_0)$; let us temporarily denote it by $\cA$.  Moreover, the equality
  \[
  \mathbf{1}_{Z(a,b)}=\mathbf{1}_{Z(\emptyset,a)}^* \mathbf{1}_{Z(\emptyset,b)}
  \]
  shows that $\cA$ is generated as a $*$-algebra by the elements
  \[
    \mathbf{1}_{Z(\emptyset,a)}, \qquad \text{with }\nu\in\bPIplus,\ a\in\cB(\nu).
  \]
  In light of Definition~\ref{def:OXIL}, the result now follows directly from the isomorphism~\eqref{eq:isomorphism}.
\end{proof}

\subsection{The crystal groupoid of a flag variety}

Lemma~\ref{lem:flag_subalgebra} realises the crystallised function algebra $\cO(X^{I,L}_0)$ of a flag variety as a $*$-subalgebra of a groupoid algebra.  But our goal is to realise $\cO(X^{I,L}_0)$ as a groupoid algebra in its own right.  In this section, we define a crystal groupoid $\GXIL\rightrightarrows\crystal_I(\infty)$ associated to any Poisson homogeneous space $X^{I,L}$.  We will later prove that $\cO(X^{I,L}_0) \cong \Cc^\infty(\GXIL)$.

\begin{definition}
  \label{def:B_I}
  Let $I\subseteq\Simples$.  We define the restricted semi-infinite crystal $\crystal_I(\infty)$ as the inverse limit
    \begin{equation}
      \label{eq:B_I}
      \crystal_I(\infty) = \underset{\lambda\in\bPIplus}{\underleftarrow{\lim}} \cB(\lambda)
    \end{equation}
  over all finite crystals with highest weights in $\bPIplus$,
  with connecting maps $\sfR_\mu:\cB(\lambda) \to \cB(\mu)$ for $\lambda\geq\mu$.  As for $\crystal(\infty)$, it is endowed with the inverse limit topology.
\end{definition}

One can realise the elements $x\in \crystal_I(\infty)$ as semi-infinite Cartan words
\begin{equation}
\label{eq:I-word}
  x=\ldots x_3 x_2 x_1 x_0 \ \in\ \prod_{j=0}^\infty \cB(\mu_j),
\end{equation}
just as we did for $\crystal(\infty)$ in Sections \ref{sec:words}~and~\ref{sec:braiding}.  The only difference is that the weights $\mu_j$ must belong to $\bPIplus$ and satisfy
\begin{equation}
\label{eq:I-cofinal}
  \forall \lambda \in \bPIplus, \quad \exists n\in\NN \text{ such that }
  \lambda \leq \sum_{j=0}^n \mu_j.
\end{equation}

The restricted semi-infinite crystal $\crystal_I(\infty)$ admits a restricted action
\begin{equation}
  \label{eq:I-shift}
  \sfS : \bPIplus \curvearrowright \crystal_I(\infty)
\end{equation}
of the semigroup $\bPIplus$ as in Section~\ref{sec:action}.

\begin{definition}
  \label{def:GIL}
  Let $X^{I,L}$ be the Poisson homogeneous space of~\eqref{eq:XIL}, where $I\subseteq\Simples$ and $L\leq\bPI$.  The crystal groupoid of $X^{I,L}$ is the Deaconu-Renault groupoid $\GXIL\rightrightarrows\crystal_I(\infty)$ defined by
  \begin{multline}
    \label{eq:GXIL}
    \GXIL = \{(x,\mu-\nu,y) \in \crystal_I(\infty) \times \bPI \times \crystal_I(\infty) \mid \\
      \sfS_\mu(x) = \sfS_\nu(y) \text{ with } \mu,\nu\in\bPIplus,\ \mu-\nu\in L\},
  \end{multline}
  with groupoid operations analogous to those of Definition~\ref{def:crystal_groupoid}.
  It is given the topology generated by the base of clopen sets
  \begin{multline}
  \label{eq:Z_I}
    Z_I(a,b) = \{(za,\mu-\nu,zb) \in \GXIL \mid\\
     z\in\crystal_I(\infty) \text{ such that } za,zb\in\crystal_I(\infty)\},
  \end{multline}
  for all $a\in\cB(\mu)$, $b\in\cB(\nu)$ with $\mu,\nu\in\bPIplus$ and $\mu-\nu\in L$.

  When $L=\mathbf{0}$ or $L=\bPI$ we denote $\GXIL$ by $\GXI$ or $\GYI$, respectively, \emph{cf.}~\eqref{eq:XI} and~\eqref{eq:YI}.
\end{definition}

The next three subsections are devoted to proving the following generalisation of Theorem~\ref{thm:isomorphism}.

\begin{theorem}
\label{thm:isomorphism2}
Let $I\subseteq\Simples$ and $L\leq\bPI$.  We have graded $*$-isomorphisms
\[
  \cO(X^{I,L}_0) \cong \Cc^\infty(\GXIL) ,\qquad
  C(X^{I,L}_0) \cong C^*(\GXIL),
\]
where $\cO(X^{I,L}_0)$ and $C(X^{I,L}_0)$ are equipped with the opposite $\bP$-grading from that of~\eqref{eq:OKq_gauge_action}.
\end{theorem}

\subsection{Orbits and isotropy groups of the crystal groupoid of a flag variety}

Before proving the isomorphism theorem, Theorem~\ref{thm:isomorphism2}, it will be helpful to specify the tail-equivalence classes of the restricted crystal $\crystal_I(\infty)$.  This will yield the same kind of structure theorems for the crystal groupoid of $X^{I,L}$ as we had for the crystal groupoid of $K$ in Theorems~\ref{thm:orbits} and~\ref{thm:orbit_topology}.

Just as for $\crystal(\infty)$, we can define extremal elements for the restricted semi-infinite crystals $\crystal_I(\infty)$.  Let $I\subseteq\Simples$.  Let $\mu\in\bPIplus$ have support $I$, meaning that $\mu\geq\varpi_i$ for every $i\in I$.   As in Section~\ref{sec:Cartan_order}, we continue to use the notation $W_I$ for the Weyl group stabiliser of $\mu$, and $[w]_I$ for the coset of $w\in W$ in $W/W_I$.

\begin{definition}
  \label{b_wI}
   For $w\in W$, we define the extremal element
  \begin{equation}
    \label{eq:b_wI}
    b_{[w]_I}(\infty) = b_w(\mu)^\infty \ \in \crystal_I(\infty).
  \end{equation}
\end{definition}

The fact that $b_w(\mu)^\infty$ is indeed a Cartan word in $\crystal_I(\infty)$ follows from Corollary~\ref{cor:BxB_decreasing}.  Note that, as the notation suggests,
\[
  b_{[w]_I}(\infty) = b_{[w']_I}(\infty) \quad \Leftrightarrow \quad [w]_I = [w']_I.
\]

More generally, a semi-infinite word
\[
  \ldots x_3 x_2 x_1 x_0 \in \prod_{j\in\NN}\cB(\mu)
\]
defines an element of the restricted semi-infinite crystal $\crystal_I(\infty)$ if and only if 
$x_{j+1}\otimes x_j$ is in the Cartan component for all $j\in\NN$.
Therefore, by Corollary \ref{cor:BxB_decreasing}, every such semi-infinite Cartan word eventually stabilises to an extremal element, \emph{i.e.}
\[
  x_j = b_{w}(\mu) \quad \text{for all sufficiently large $j$}.
\]
We obtain the following analogue of Theorem~\ref{thm:orbits}, Theorem~\ref{thm:orbit_topology} and Corollary~\ref{cor:isotropy}.

\begin{theorem}
  \label{thm:GXIL_orbits}
  Fix $I\subseteq\Simples$ and $L\leq\bPI$.
  The set
  \[
    \{ b_{[w]_I}(\infty) \mid [w]_I \in W/W_I \}
  \]
  is a complete set of orbit representatives for the crystal groupoid $\GXIL$.  The resulting bijection
  \[
    W/W_I \to \crystal_I(\infty)/\GXIL , \qquad
    [w]_I \mapsto [b_{[w]_I}(\infty)],
  \]
  is a homeomorphism with the quotient topology on the space of $\GXIL$-orbits and the Bruhat-order topology on $W/W_I$.
  The isotropy group at any unit in $\GXIL$ is isomorphic to $L$.
\end{theorem}

\begin{proof}
  We have just shown that every element $x\in\crystal_I(\infty)$ is tail equivalent to a unique extremal element, so $(x,0,b_{[w]_I}(\infty))\in\GXIL$ for some $w\in W$.
  For the topology on the orbit space, the argument presented in the proof of Theorem~\ref{thm:orbit_topology} applies almost without change, and we deduce that
  \[
    [b_{[v]_I}(\infty)] \in \overline{[b_{[w]_I}(\infty)]} \quad
    \Leftrightarrow \quad
    [v]_I \leq [w]_I.
  \]
  Finally, since $b_{[w]_I}(\infty)$ is shift-invariant, the isotropy group at $b_{[w]_I}(\infty)$ is $L$.
\end{proof}

\subsection{Semi-infinite truncations and semi-infinite shifts}

\begin{definition}
\label{def:R_I}
  Let $I\subseteq\Simples$.  We define the \emph{semi-infinite truncation}
  \[
    \sfR_I : \crystal(\infty) \to \crystal_I(\infty); \qquad
    (x_\lambda)_{\lambda\in\bP^+} \mapsto (x_\lambda)_{\lambda\in\bPIplus}.
  \]
\end{definition}

Note that $\sfR_I$ is continuous.
The following relations follow from~\eqref{eq:RR} and~\eqref{eq:SR}.

\begin{proposition}
\label{prop:RS-semi-infinite}
For any $\nu\in\bPIplus$ we have
\begin{align}
  \label{eq:RR_I}
  &\sfR_\nu \circ \sfR_I = \sfR_\nu : \crystal(\infty) \to \cB(\nu),\\
  \label{eq:R_IS_lambda}
  &\sfS_\nu\circ\sfR_I = \sfR_I\circ\sfS_\nu :\crystal(\infty) \to \crystal_I(\infty).
\end{align}
\end{proposition}

Just as a semi-infinite crystal $x\in\crystal(\infty)$ can be decomposed by a finite truncation $\sfR_\nu$ and shift $\sfS_\nu$,
\begin{equation*}
   \crystal(\infty) \to \crystal(\infty) \times \cB(\nu), \qquad
   x \mapsto (\sfS_\nu(x),\sfR_\nu(x)),
\end{equation*}
so too can we give a decomposition of $x\in\crystal(\infty)$ via a semi-infinite truncation $\sfR_I$ and semi-infinite shift $\sfS_I$.  In the following, we will view $\bPIplus$ as a directed system via the usual algebraic ordering.

\begin{proposition}
\label{prop:infinite_shift}
  Fix $I\subseteq\Simples$. For any semi-infinite crystal $x\in\crystal(\infty)$, the net of shifts
  \[
    (\sfS_\lambda(x))_{\lambda\in\bPIplus},
  \]
  is eventually constant, in the sense that there exists $\lambda_0\in\bPIplus$ such that
  \[
  \sfS_\lambda(x) = \sfS_{\lambda_0}(x) \qquad \text{for all }\lambda\in\bPIplus,\  \lambda\geq\lambda_0.
  \]
\end{proposition}

\begin{proof}
Thanks to Theorem~\ref{thm:orbits} and Remark \ref{rmk:shift-invariant}, we can write
\[
  x= b_w(\infty) x_0
\]
for some $x_0\in\cB(\mu)$ and some $w\in W$.   Then for every $\lambda\in\bPIplus$, writing
\[
  x_\lambda := \sfR_\mu\circ\sfS_\lambda(x) = \sfS_\lambda(b_w(\lambda)x_0) \in \cB(\mu),
\]
we have
\[
  \sfS_\lambda(x) = b_w(\infty) x_\lambda.
\]
For $\nu\in\bPIplus$, the elements $x_\lambda$ and $x_{\lambda+\nu}$ are related by
\(
  x_{\lambda+\nu} = \sfS_\nu(b_w(\nu)x_\lambda)
\),
which is equivalent to the Cartan braiding relation
\[
  \sigma (b_w(\nu) \otimes x_\lambda) = x_{\lambda+\nu} \otimes b'
\]
for some $b'\in\cB(\nu)$.  By the ordering property of the Cartan braiding, see Lemma~\ref{lem:braiding_order}, the net $(x_\lambda)_{\lambda\in\bPIplus}$ is decreasing, in the sense that
\[
  \wt(x_{\lambda+\nu}) \preceq \wt(x_\lambda), \qquad \forall \nu\in\bPIplus.
\]
Moreover, whenever these weights are equal we have $x_{\lambda+\nu}=x_{\lambda}$.  It follows that $(x_\lambda)_{\lambda\in\bPIplus}$ is eventually constant, and so there is $\lambda_0\in\bPIplus$ such that
\[
  \sfS_{\lambda}(x) = b_w(\infty) x_{\lambda_0}, \qquad
  \text{for all }\lambda\in\bPIplus,\ \lambda\geq\lambda_0.
  \qedhere
\]
\end{proof}

\begin{definition}
\label{def:infinite_shift}
  Let $I\subseteq\Simples$.  We define the \emph{semi-infinite shift}
  \[
    \sfS_I : \crystal(\infty) \to \crystal(\infty); \qquad
    x \mapsto \lim_{\lambda\in\bPIplus} \sfS_\lambda(x),
  \]
  where the limit refers to the ultimately constant value which is given by Proposition~\ref{prop:infinite_shift}.
\end{definition}

The analogue of Equation \eqref{eq:SS} is that for any $\nu\in\bPIplus$,
\begin{equation}
\label{eq:SIS}
  \sfS_I\circ \sfS_\nu = \sfS_\nu\circ\sfS_I = \sfS_I : \crystal(\infty) \to \crystal(\infty).
\end{equation}

For instance, if $I=\Simples$, the semi-infinite shift $\sfS_{\Simples}$ sends $x$ to the unique shift-invariant element in its $\GK$-orbit; see Remark~\ref{rmk:shift-invariant}.  The semi-infinite truncation $\sfR_\Simples:\crystal(\infty)\to\crystal(\infty)$ is the identity.  Note that this example shows that $\sfS_I$ is not usually continuous, unlike $\sfR_I$.

Combining a semi-infinite shift and truncation, we get the decomposition previously alluded to.

\begin{proposition}
\label{prop:SxR}
Let $I\subseteq\Simples$.
\begin{itemize}
  \item[(i)]
The map
\begin{align}
  \label{eq:SxR}
  (\sfS_I,\sfR_I) : \crystal(\infty) & \to \crystal(\infty) \times \crystal_I(\infty); \\
  x &\mapsto (\sfS_{I}(x), \sfR_{I}(x))
  \nonumber
\end{align}
is injective.

\item[(ii)]
If a pair $(z,y) \in \crystal(\infty)\times\crystal_I(\infty)$ belongs to the image of the decomposition map $(\sfS_I,\sfR_I)$ then so does $(z,y')$, for any $y'\in\crystal_I(\infty)$ which is tail equivalent to $y$, meaning that $\sfS_\mu(y) = \sfS_\nu(y')$ for some $\mu,\nu\in\bPIplus$.
\end{itemize}
\end{proposition}

\begin{proof}
  \begin{itemize}
    \item[\emph{(i)}]
  Suppose $x,x'\in\crystal(\infty)$ have the same image under $(\sfS_I,\sfR_I)$.  Put
  \[
    z=\sfS_{I}(x) = \sfS_{I}(x').
  \]
  Fixing $\lambda_0\in\bPIplus$ sufficiently large as in Proposition~\ref{prop:infinite_shift}, we get
  \[
    x= zy_0, \qquad x' =  zy'_0
  \]
  for some $y_0,y'_0 \in \cB(\lambda_0)$.  Using $\sfR_{I}(x)=\sfR_{I}(x')$, we get $y_0=y'_0$.

  \item[\emph{(ii)}]
  Let $(z,y)$, $(z,y')$ and $\mu,\nu$ be as in the statement, and suppose that there exists $x\in\crystal(\infty)$ such that
  \[
    (\sfS_I(x),\sfR_I(x)) = (z,y).
  \]
  Fix $\lambda\in\bP^+$. Decompose $\sfR_{\lambda+\nu+\mu}(x)$ as a Cartan word
  \begin{equation}
    \label{eq:SxR1}
    \sfR_{\lambda+\nu+\mu}(x) = x_2 x_1 x_0 \quad\in \cB(\lambda)\times\cB(\nu)\times\cB(\mu).
  \end{equation}
  Then
  \[
    \sfR_{\nu+\mu}(y) = \sfR_{\nu+\mu}(x) = x_1x_0,
  \]
  and since $\sfS_\mu(y) = \sfS_\nu(y')$ we get $\sfR_{\nu}(\sfS_\nu(y')) = \sfR_{\nu}(\sfS_\mu(y))= x_1$ and hence
  \begin{equation}
    \label{eq:SxR2}
    \sfR_{2\nu}(y') = x_1x'_0
  \end{equation}
  for some $x'_0\in\cB(\nu)$.  From~\eqref{eq:SxR1} and~\eqref{eq:SxR2} we deduce that both $x_2x_1$ and $x_1x'_0$ are Cartan words.  But then
  \begin{align*}
    \sigma_{\cB(\lambda+\nu),\cB(\nu)}&((x_2\otimes x_1) \otimes x'_0) \\
    &= (\sigma_{\cB(\lambda),\cB(\nu)}\otimes\id)(\id\otimes\sigma_{\cB(\nu),\cB(\nu)})(x_2\otimes x_1\otimes x'_0) \\
    &= (\sigma_{\cB(\lambda),\cB(\nu)}\otimes\id)(x_2\otimes x_1\otimes x'_0) \\
    &\neq 0,
  \end{align*}
  and hence $x_2x_1x'_0 = \sfR_{\lambda+\nu}(\sfS_{\mu}(x))x'_0$ is a Cartan word, for any $\lambda\in\bP^+$.  Therefore, the semi-infinite word
  \[
    x' = \sfS_{\mu}(x)x'_0
  \]
  is a Cartan word and it satisfies
  \begin{align*}
    \sfS_I(x') &= \sfS_I(x) = z, \\
    \sfR_I(x') &= \sfR_I(\sfS_\mu(x))x'_0 = \sfS_\mu(y)x'_0 = y'.
  \end{align*}
  \end{itemize}
  This completes the proof.
\end{proof}

\begin{corollary}
\label{cor:R_I_surjective}
  The map $\sfR_I:\crystal(\infty) \to \crystal_I(\infty)$ is surjective.
\end{corollary}

\begin{proof}
  Let $x'\in\crystal_I(\infty)$.  By Theorem~\ref{thm:GXIL_orbits}, $x'$ is tail-equivalent to some extremal element $b_{[w]_I}(\infty)\in\crystal_I(\infty)$, with $w\in W$.  But note that
  \[
    (\sfS_I(b_w(\infty)),\sfR_I(b_w(\infty))) = (b_w(\infty),b_{[w]_I}(\infty)).
  \]
  Therefore, by Proposition~\ref{prop:SxR}\emph{(ii)}, there is some $x\in\crystal(\infty)$ such that
  \[
    (\sfS_I(x),\sfR_I(x)) = (b_w(\infty),x').
  \]
  This completes the proof.
\end{proof}

\subsection{Proof of the isomorphism for crystallised flag varieties}
\label{sec:proof}

We will now complete the proof of Theorem~\ref{thm:isomorphism2} with a series of lemmas.

\begin{lemma}
  \label{lem:GIK}
  Fix $I\subseteq\Simples$.  The set
  \begin{multline}
  \label{eq:GI}
    \GIK = \{ (x,\mu-\nu,y) \in \GK \mid
    \mu,\nu\in\bPIplus,\ \sfS_\mu(x) = \sfS_\nu(y)\}
  \end{multline}
  is an open subgroupoid of $\GK$, and the map
  \begin{equation}
  \label{eq:pi_I}
    \pi_I = \sfR_{I} \times \id \times \sfR_{I}
     : \GIK \to \GYI
  \end{equation}
  is a continuous surjective groupoid morphism.
\end{lemma}

\begin{proof}
  It is immediate that $\GIK$ is a subgroupoid.  To see that it is open, let $(x,\lambda,y)\in\GIK$.  Then
\[
  \sfS_{\mu}(x) = \sfS_{\nu}(y),
\]
for some $\mu,\nu\in\bPIplus$ with $\mu-\nu=\lambda$.  Putting $x_0 = \sfR_{\mu}(x)$ and $y_0 = \sfR_{\nu}(y)$, we have
\[
  (x,\lambda,y) \in Z(x_0,y_0) \subseteq \GIK,
\]
which is an open neighbourhood in $\GK$.

It follows from the definition of $\GIK$ that $\pi_I(\GIK)\subseteq\GYI$, and it is clearly a groupoid morphism.  To show that it is continuous, we will prove that
  \begin{equation}
  \label{eq:pi_I_continuous}
    \pi_I^{-1}(Z_I(a,b)) = Z(a,b),
  \end{equation}
  where $Z_I(a,b)$ is any open set from the neighbourhood basis of Equation \eqref{eq:Z_I} for $a\in\cB(\mu)$, $b\in\cB(\nu)$ with $\mu,\nu\in\bPIplus$.
  Certainly $\pi_I(Z(a,b)) \subseteq Z_I(a,b)$.  Conversely, suppose that $(x,\lambda,y)\in\GIK$ with $\pi_I(x,\lambda,y)\in Z_I(a,b)$.  Then we can write
  \[
    x = x'a, \quad y=y'b
  \]
  for some $x',y'\in\crystal(\infty)$ with
  \[
  \sfR_I(x') = \sfR_I(y').
  \]
  But also,
  \[
  \sfS_I(x') = \sfS_I(x) = \sfS_I(y)= \sfS_I(y'),
  \]
  so by Proposition~\ref{prop:SxR}, we have $x'=y'$.  Therefore $(x,\lambda,y)\in Z(a,b)$, which proves the claim.

Finally, for surjectivity, let $(x',\lambda,y')\in\GYI$.  By the surjectivity of $\sfR_I$ in Corollary~\ref{cor:R_I_surjective}, there is $x\in\crystal(\infty)$ with $\sfR_I(x)=x'$.  Since $x'$ and $y'$ are tail equivalent in $\crystal_I(\infty)$, Proposition~\ref{prop:SxR}\emph{(ii)} implies that there is also $y\in\crystal(\infty)$ such that $\sfR_I(y)=y'$ and $\sfS_I(x)=\sfS_I(y)$.  The latter equation means that $\sfS_\mu(x) = \sfS_\nu(y)$ for all $\mu,\nu\in\bPIplus$ sufficiently large, and we may take them so that $\lambda = \mu-\nu$.  It follows that $(x,\lambda,y)\in\cG_I(K)$, which completes the proof.
\end{proof}

\begin{remark}
  \label{rmk:GIK_characterization}
An element $(x,\lambda,y)\in\GK$ belongs to $\GIK$ if and only if $\lambda\in\bPI$ and $\sfS_I(x) = \sfS_I(y)$.
\end{remark}

\begin{lemma}
  \label{lem:GIK_decomposition}
  Let $\GIK$ be the open subgroupoid of Lemma~\ref{lem:GIK}, take $(x,\lambda,y) \in \GIK$, and write $\pi_I(x,\lambda,y) = (x',\lambda,y')\in\GYI$.  If $z'\in\crystal_I(\infty)$ and $\mu,\nu\in\bPI$ are such that
  \[
     (x',\lambda,y') = (x',\mu,z')(z',\nu,y'),
  \]
  then there exists a unique preimage $z\in\sfR_I^{-1}(z')$ such that $(x,\mu,z)$ and $(z,\nu,y)$ are in $\GIK$.
\end{lemma}

\begin{proof}
  By the definition of $\GYI$, we have
  \[
    \sfS_{\mu'}(x') = \sfS_{\mu''}(z') \text{ and }
    \sfS_{\nu'}(z') = \sfS_{\nu''}(y'),
  \]
  for some $\mu',\mu'',\nu',\nu''\in\bPIplus$ with $\mu=\mu'-\mu''$ and $\nu=\nu'-\nu''$.
  By Proposition~\ref{prop:SxR}\emph{(ii)}, there exists $z\in\crystal(\infty)$ such that
  \[
    (\sfS_I(z),\sfR_I(z)) = (\sfS_I(x), z').
  \]
  By Remark~\ref{rmk:GIK_characterization}, $(x,\mu,z)$ and $(z,\nu,y)$ are in $\GIK$.  This proves existence.  Uniqueness follows from the injectivity of $(\sfR_I,\sfS_I)$, see Proposition~\ref{prop:SxR}\emph{(i)}.
\end{proof}

\begin{lemma}
  \label{lem:pi_I_star}
Fix $I\subseteq\Simples$.  The map
\begin{equation}
  \label{eq:pi_I_star}
  \Cc^\infty(\GYI)
   \stackrel{\pi_I^*}{\longrightarrow} \Cc^\infty(\GIK) \into \Cc^\infty(\GK),
\end{equation}
where the first arrow is the pullback by $\pi_I$ from Equation~\eqref{eq:pi_I} and the second map is extension by zero, is an injective graded $*$-algebra morphism.  Its image is isomorphic to $\cO(Y^I_0)$ under the isomorphism of Lemma~\ref{lem:flag_subalgebra}.
\end{lemma}

\begin{proof}
  Extension by zero from the open subgroupoid $\GIK$ to $\GK$ is certainly a graded $*$-morphism, so it suffices to consider the map $\pi_I^*$.  To find the image of $\pi_I^*$, first recall that $\Cc^\infty(\GYI)$ is spanned by the indicator functions $\mathbf{1}_{Z_I(a,b)}$, where $Z_I(a,b)$ are the clopen subsets of~\eqref{eq:Z_I}, for $a\in\cB(\mu)$, $b\in\cB(\nu)$ with $\mu,\nu\in\bPIplus$.
  But $\pi_I^*\mathbf{1}_{Z_I(a,b)} = \mathbf{1}_{Z(a,b)}$ thanks to Equation \eqref{eq:pi_I_continuous}.  This confirms that its image is contained in $\Cc^\infty(\GIK)$ and under the isomorphism of Theorem \ref{thm:isomorphism} it corresponds to $\cO(Y^I_0)$ by Lemma \ref{lem:flag_subalgebra}.
  
  The fact that $\pi_I^*$ is $*$-preserving and graded is straightforward from the definitions.  It is an algebra morphism by the unique lifting property of Lemma~\ref{lem:GIK_decomposition}, and it is injective by the surjectivity of $\pi_I$ from Lemma~\ref{lem:GIK}.
  This completes the proof.
\end{proof}

\begin{proof}[Proof of Theorem~\ref{thm:isomorphism2}]
  The map described in Lemma~\ref{lem:pi_I_star} is a graded $*$-isomorphism
  \begin{equation*}
    \cO(Y^I_0) \cong \Cc^\infty(\GYI),
  \end{equation*}
  where $\cO(Y^I_0)$ has the opposite $\bP$-grading.  This is Theorem~\ref{thm:isomorphism2} in the case where $L=\bPI$.  Since it preserves the $\bP$-grading, it restricts to an isomorphism between the subalgebras
  \[
    \cO(X^{I,L}_0) \cong \Cc^\infty(\GXIL),
  \]
  for every $L\leq\bPI$.  The isomorphism of the $C^*$-completions can be proven similarly; we leave the details to the reader.
\end{proof}


\printbibliography

@preamble{"\def\polhk#1{\setbox0=\hbox{#1}{\ooalign{\hidewidth     \lower1.5ex\hbox{`}\hidewidth\crcr\unhbox0}}} "}

@article{MatYun:crystals,
	author = {Matassa, Marco and Yuncken, Robert},
	doi = {10.1515/crelle-2023-0047},
	fjournal = {Journal f{\"u}r die Reine und Angewandte Mathematik},
	issn = {0075-4102},
	journal = {J. Reine Angew. Math.},
	language = {English},
	pages = {173--221},
	title = {Crystal limits of compact semisimple quantum groups as higher-rank graph algebras},
	volume = {802},
	year = {2023}}

@article{Littelmann,
	author = {Littelmann, Peter},
	doi = {10.1006/jabr.1995.1175},
	fjournal = {Journal of Algebra},
	issn = {0021-8693},
	journal = {J. Algebra},
	language = {English},
	number = {1},
	pages = {65--87},
	title = {Crystal graphs and {Young} tableaux},
	volume = {175},
	year = {1995}}

@incollection{Sheu:CQG_lectures,
	author = {Sheu, Albert J. L.},
	booktitle = {Quantum groups and quantum spaces. Lectures delivered during the minisemester, Warsaw, Poland, December 1, 1995},
	language = {English},
	pages = {41--50},
	publisher = {Warszawa: Polish Academy of Sciences, Inst. of Mathematics},
	title = {Groupoids and compact quantum groups},
	year = {1997}}

@article{Sheu:CQGs,
	author = {Sheu, Albert J. L.},
	doi = {10.1006/jfan.1996.2999},
	fjournal = {Journal of Functional Analysis},
	issn = {0022-1236},
	journal = {J. Funct. Anal.},
	language = {English},
	number = {2},
	pages = {371--393},
	title = {Compact quantum groups and groupoid {{\(C^*\)}}-algebras},
	volume = {144},
	year = {1997}}

@incollection{Sheu:projective_spaces,
	author = {Sheu, Albert Jeu-Liang},
	booktitle = {Operator algebras and operator theory. Proceedings of the international conference, Shanghai, China, July 4--9, 1997},
	isbn = {0-8218-1093-6},
	language = {English},
	pages = {341--350},
	publisher = {Providence, RI: American Mathematical Society},
	title = {Groupoid approach to quantum projective spaces},
	year = {1998}}

@article{Sheu:spheres,
	author = {Sheu, Albert J. L.},
	fjournal = {The Quarterly Journal of Mathematics. Oxford Second Series},
	issn = {0033-5606},
	journal = {Q. J. Math., Oxf. II. Ser.},
	language = {English},
	number = {192},
	pages = {503--510},
	title = {Quantum spheres as groupoid {{\(C^*\)}}-algebras},
	volume = {48},
	year = {1997}}

@article{SimWil:primitive_ideals,
	author = {Sims, Aidan and Williams, Dana P.},
	doi = {10.1007/s10468-015-9573-4},
	fjournal = {Algebras and Representation Theory},
	issn = {1386-923X},
	journal = {Algebr. Represent. Theory},
	language = {English},
	number = {2},
	pages = {255--276},
	title = {The primitive ideals of some {\'e}tale groupoid {{\(C^\ast\)}}-algebras},
	volume = {19},
	year = {2016}}

@article{Paterson,
	author = {Paterson, Alan L. T.},
	fjournal = {Journal of Operator Theory},
	issn = {0379-4024},
	journal = {J. Oper. Theory},
	language = {English},
	number = {3},
	pages = {645--662},
	title = {Graph inverse semigroups, groupoids and their {{\(C^*\)}}-algebras},
	volume = {48},
	year = {2002}}

@incollection{Renault:Cuntz-like,
	author = {Renault, Jean},
	booktitle = {Operator theoretical methods. Proceedings of the 17th international conference on operator theory, Timi\c{s}oara, Romania, June 23--26, 1998},
	isbn = {973-99097-2-8},
	language = {English},
	pages = {371--386},
	publisher = {Bucharest: The Theta Foundation},
	title = {Cuntz-like algebras},
	year = {2000}}

@article{Kas:Demazure,
	author = {Kashiwara, Masaki},
	doi = {10.1215/S0012-7094-93-07131-1},
	fjournal = {Duke Mathematical Journal},
	issn = {0012-7094},
	journal = {Duke Math. J.},
	language = {English},
	number = {3},
	pages = {839--858},
	title = {The crystal base and {Littelmann}'s refined {Demazure} character formula},
	volume = {71},
	year = {1993}}

@article{AssDraGon,
	author = {Assaf, Sami and Dranowski, Anne and Gonz{\'a}lez, Nicolle},
	doi = {10.1007/s10468-023-10231-z},
	fjournal = {Algebras and Representation Theory},
	issn = {1386-923X},
	journal = {Algebr. Represent. Theory},
	language = {English},
	number = {1},
	pages = {627--638},
	title = {Extremal tensor products of {Demazure} crystals},
	volume = {27},
	year = {2024}}

@misc{Matassa:some_results,
	arxiv = {arXiv:2504.00684},
	author = {Marco Matassa},
	howpublished = {Preprint, {arXiv}:2504.00684 [math.{CO}] (2026)},
	title = {Some results on the higher-rank graphs associated to crystals of semisimple {Lie} algebras},
	url = {https://arxiv.org/abs/2504.00684},
	year = {2026}}

@book{BjoBre,
	author = {Bj{\"o}rner, Anders and Brenti, Francesco},
	doi = {10.1007/3-540-27596-7},
	fseries = {Graduate Texts in Mathematics},
	isbn = {3-540-44238-3},
	issn = {0072-5285},
	language = {English},
	publisher = {New York, NY: Springer},
	series = {Grad. Texts Math.},
	title = {Combinatorics of {Coxeter} groups},
	volume = {231},
	year = {2005}}

@article{VakSoi:SUn,
	author = {Vaksman, L. L. and Sojbel'man, Ya. S.},
	fjournal = {Leningrad Mathematical Journal},
	issn = {1048-9924},
	journal = {Leningr. Math. J.},
	language = {English},
	number = {5},
	pages = {1023--1042},
	title = {The algebra of functions on the quantum group {{\(SU(n+1)\)}}, and odd- dimensional quantum spheres},
	volume = {2},
	year = {1990}}

@book{Renault:groupoid_approach,
	author = {Renault, Jean},
	fseries = {Lecture Notes in Mathematics},
	issn = {0075-8434},
	language = {English},
	publisher = {Springer, Cham},
	series = {Lect. Notes Math.},
	title = {A groupoid approach to {C}*-algebras},
	volume = {793},
	year = {1980}}

@article{Deaconu,
	author = {Deaconu, Valentin},
	doi = {10.2307/2154972},
	fjournal = {Transactions of the American Mathematical Society},
	issn = {0002-9947},
	journal = {Trans. Am. Math. Soc.},
	language = {English},
	number = {5},
	pages = {1779--1786},
	title = {Groupoids associated with endomorphisms},
	volume = {347},
	year = {1995}}

@article{StoDij,
	author = {Stokman, Jasper V. and Dijkhuizen, Mathijs S.},
	doi = {10.1007/s002200050613},
	fjournal = {Communications in Mathematical Physics},
	issn = {0010-3616},
	journal = {Commun. Math. Phys.},
	language = {English},
	number = {2},
	pages = {297--324},
	title = {Quantized flag manifolds and irreducible {{\(*\)}}-representations},
	volume = {203},
	year = {1999}}

@misc{Giselsson:SU3,
	archiveprefix = {arXiv},
	author = {Olof Giselsson},
	eprint = {2307.12878},
	primaryclass = {math.OA},
	title = {Quantum {$SU(3)$} as the {$C^*$}-algebra of a 2-graph},
	url = {https://arxiv.org/abs/2307.12878},
	year = {2023}}

@article{Aranda:Kumjian-Pask,
	author = {Aranda Pino, Gonzalo and Clark, John and an Huef, Astrid and Raeburn, Iain},
	doi = {10.1090/S0002-9947-2013-05717-0},
	fjournal = {Transactions of the American Mathematical Society},
	issn = {0002-9947},
	journal = {Trans. Amer. Math. Soc.},
	mrclass = {16S99 (16W50)},
	mrnumber = {3042597},
	number = {7},
	pages = {3613--3641},
	title = {Kumjian-{P}ask algebras of higher-rank graphs},
	url = {https://doi.org/10.1090/S0002-9947-2013-05717-0},
	volume = {365},
	year = {2013}}

@book{BumpSchilling,
	author = {Bump, Daniel and Schilling, Anne},
	doi = {10.1142/9876},
	isbn = {978-981-4733-44-1},
	mrclass = {05-01 (05E10 14T05 17B10)},
	mrnumber = {3642318},
	note = {Representations and combinatorics},
	pages = {xii+279},
	publisher = {World Scientific Publishing Co. Pte. Ltd., Hackensack, NJ},
	title = {Crystal bases},
	url = {https://doi.org/10.1142/9876},
	year = {2017}}

@article{EtiSchSol,
	author = {Etingof, Pavel and Schedler, Travis and Soloviev, Alexandre},
	doi = {10.1215/S0012-7094-99-10007-X},
	fjournal = {Duke Mathematical Journal},
	issn = {0012-7094},
	journal = {Duke Math. J.},
	language = {English},
	number = {2},
	pages = {169--209},
	title = {Set-theoretical solutions to the quantum {Yang}-{Baxter} equation},
	volume = {100},
	year = {1999},
	zbl = {0969.81030},
	zbmath = {1425260}}

@unpublished{GirPal,
	author = {Giri, Manabendra and Pal, Arup Kumar},
	note = {preprint},
	title = {Quantized function algebras at $q=0$: type $A_n$ case},
	url = {https://arxiv.org/pdf/2203.14665.pdf},
	year = {2022}}

@book{HonKan,
	author = {Hong, Jin and Kang, Seok-Jin},
	fseries = {Graduate Studies in Mathematics},
	isbn = {0-8218-2874-6},
	issn = {1065-7338},
	language = {English},
	publisher = {Providence, RI: American Mathematical Society (AMS)},
	series = {Grad. Stud. Math.},
	title = {Introduction to quantum groups and crystal bases},
	volume = {42},
	year = {2002},
	zbl = {1134.17007},
	zbmath = {1757924}}

@article{HonSzy:lens,
	author = {Hong, Jeong Hee and Szyma{\'n}ski, Wojciech},
	doi = {10.2140/pjm.2003.211.249},
	fjournal = {Pacific Journal of Mathematics},
	issn = {1945-5844},
	journal = {Pac. J. Math.},
	language = {English},
	number = {2},
	pages = {249--263},
	title = {Quantum lens spaces and graph algebras.},
	volume = {211},
	year = {2003},
	zbl = {1058.46047},
	zbmath = {2101109}}

@article{Kashiwara:crystal2,
	author = {Kashiwara, M.},
	doi = {10.1215/S0012-7094-91-06321-0},
	fjournal = {Duke Mathematical Journal},
	issn = {0012-7094},
	journal = {Duke Math. J.},
	mrclass = {17B37 (17B10 17B67)},
	mrnumber = {1115118},
	mrreviewer = {Kailash C. Misra},
	number = {2},
	pages = {465--516},
	title = {On crystal bases of the {$Q$}-analogue of universal enveloping algebras},
	url = {https://doi.org/10.1215/S0012-7094-91-06321-0},
	volume = {63},
	year = {1991}}

@article{KumPas,
	author = {Kumjian, Alex and Pask, David},
	fjournal = {The New York Journal of Mathematics},
	issn = {1076-9803},
	journal = {New York J. Math.},
	language = {English},
	pages = {1--20},
	title = {Higher rank graph {{\(C^*\)}}-algebras},
	volume = {6},
	year = {2000},
	zbl = {0946.46044},
	zbmath = {1409630}}

@article{KumPasRae,
	author = {Kumjian, Alex and Pask, David and Raeburn, Iain},
	doi = {10.2140/pjm.1998.184.161},
	fjournal = {Pacific Journal of Mathematics},
	issn = {0030-8730,1945-5844},
	journal = {Pacific J. Math.},
	mrclass = {46L35 (05C20 22A22)},
	mrnumber = {1626528},
	mrreviewer = {Berndt\ Brenken},
	number = {1},
	pages = {161--174},
	title = {Cuntz-{K}rieger algebras of directed graphs},
	url = {https://doi.org/10.2140/pjm.1998.184.161},
	volume = {184},
	year = {1998}}

@article{KumPasRaeRen,
	author = {Kumjian, Alex and Pask, David and Raeburn, Iain and Renault, Jean},
	doi = {10.1006/jfan.1996.3001},
	fjournal = {Journal of Functional Analysis},
	issn = {0022-1236,1096-0783},
	journal = {J. Funct. Anal.},
	mrclass = {46L05 (46L55)},
	mrnumber = {1432596},
	mrreviewer = {Masatoshi\ Enomoto},
	number = {2},
	pages = {505--541},
	title = {Graphs, groupoids, and {C}untz-{K}rieger algebras},
	url = {https://doi.org/10.1006/jfan.1996.3001},
	volume = {144},
	year = {1997}}

@article{NesTus:functions,
	author = {Neshveyev, Sergey and Tuset, Lars},
	doi = {10.1007/s00220-012-1455-6},
	fjournal = {Communications in Mathematical Physics},
	issn = {0010-3616},
	journal = {Commun. Math. Phys.},
	language = {English},
	number = {1},
	pages = {223--250},
	title = {Quantized algebras of functions on homogeneous spaces with {Poisson} stabilizers},
	volume = {312},
	year = {2012},
	zbl = {1250.22017},
	zbmath = {6037971}}

@article{Soibelman,
	author = {Soibelman, Ya. S.},
	fjournal = {Leningrad Mathematical Journal},
	issn = {1048-9924},
	journal = {Leningr. Math. J.},
	language = {English},
	number = {1},
	pages = {161--178},
	title = {The algebra of functions on a compact quantum group, and its representations},
	volume = {2},
	year = {1991},
	zbl = {0718.46012},
	zbmath = {4184162}}

@article{Stokman:quantum_orbit_method,
	author = {Stokman, Jasper V.},
	doi = {10.4310/MRL.2003.v10.n4.a6},
	fjournal = {Mathematical Research Letters},
	issn = {1073-2780},
	journal = {Math. Res. Lett.},
	language = {English},
	number = {4},
	pages = {469--481},
	title = {The quantum orbit method for generalized flag manifolds},
	volume = {10},
	year = {2003},
	zbl = {1063.58003},
	zbmath = {2062124}}

@article{Woronowicz:pseudogroups,
	author = {Woronowicz, S. L.},
	fjournal = {Communications in Mathematical Physics},
	issn = {0010-3616},
	journal = {Comm. Math. Phys.},
	mrclass = {46L99 (22E65 58H05)},
	mrnumber = {901157},
	mrreviewer = {Pierre de la Harpe},
	number = {4},
	pages = {613--665},
	title = {Compact matrix pseudogroups},
	url = {http://projecteuclid.org/euclid.cmp/1104159726},
	volume = {111},
	year = {1987}}

\end{document}